\documentclass[12pt]{amsart}

\usepackage{graphicx}

\usepackage{subcaption} 
\usepackage{xfrac}   
\usepackage{faktor}
\usepackage{tikz-cd}
\usepackage{mathrsfs}
\usepackage{mathtools}
\usepackage{hyperref}
\usepackage{amsmath}
\usepackage{amssymb}
\hypersetup{bookmarks=true,
	unicode=true,
	colorlinks=true,
	citecolor=black,
	linkcolor=black,
	urlcolor=black,
	plainpages=false,
	pdfpagelabels=true}

 \usepackage{url}	
 \allowdisplaybreaks 

\usepackage{tikz-cd}
\usepackage{pgf}

\usepackage{xcolor}

\usepackage{comment} 

\usepackage[all]{xy}

\newtheorem{teo}{Theorem}[section]
\newtheorem{theorem}[teo]{Theorem}

\newtheorem{corollary}[teo]{Corollary}

\newtheorem{lemma}[teo]{Lemma}
\newtheorem{prop}[teo]{Proposition}
\newtheorem{proposition}[teo]{Proposition}

\theoremstyle{definition}
\newtheorem{definition}[teo]{Definition}

\newtheorem{example}[teo]{Example}

\theoremstyle{remark}
\newtheorem{remark}{Remark}

\numberwithin{figure}{section}

\newcommand{\Ba}{\mathcal{B}}

\newcommand{\bbar}{\overline}

\newcommand{\CC}{\mathcal{C}}

\newcommand{\cat}{\mathop{\mathbf{Cat}}}
\newcommand{\ccat}{\mathop{\mathrm{ccat}}}

\newcommand{\Dc}{\mathcal{D}}

\newcommand{\E}{\mathcal{E}}

\newcommand{\op}{\mathrm{op}}

\newcommand{\Sets}{\mathbf{Set}}

\newcommand{\Uc}{\mathcal{U}}

\newcommand{\Vc}{\mathcal{V}}

\newcommand{\Hom}{\mathrm{Hom}}

\begin{document}

\title[Directed Homotopy, Sectional Invariants, and Databases]{Directed Homotopy, Sectional Invariants, and Functorial Databases}
\thanks{\textbf{Funding}: This research did not receive funding.}

\author[Isaac Carcacía-Campos]{%
	Isaac Carcacía-Campos
}

 \address{%
           	Isaac Carcacía-Campos
            \\
              Departamento de Matem\'aticas, Universidade de Santiago de Compostela, 15782-SPAIN}
               \email{isaac.c.campos@usc.es}


\begin{abstract} 
A database instance on a small category may be represented as a set-valued functor or, equivalently, as a discrete opfibration. Its sections correspond to globally coherent choices of records. When no global section exists, we measure the failure of global coherence by the minimum number of subcategories on which coherent choices can be made.

Regarding natural transformations as directed homotopies, we introduce right and left directed fibrations and relate them to Grothendieck opfibrations and fibrations. We define directed versions of Lusternik-Schnirelmann category and sectional category and establish their invariance and comparison properties. Every functor admits a Grothendieck opfibration model on which directed sectional category is computed by strict local sections, together with a canonical discrete approximation obtained from connected components of comma categories.

For functorial databases, we study directed sectional category under decomposition, iteration, and data migration, and characterize initial objects of finite connected acyclic schemas through the existence of global sections of objectwise non-empty databases.
\end{abstract}



\maketitle


\section*{Introduction}

In the functorial model of databases \cite[chapter 3]{SevenSketches} a database schema is represented by a
small category \(\CC\), while an instance on that schema is a functor
\[
X\colon\CC\longrightarrow\Sets.
\] The objects of \(\CC\) represent
entity types, its morphisms represent functional attributes, \(X(c)\) is the set of objects of type \(c\) and each morphism \(X(f)\) from a morphism \(f \colon c \to d\) represents a morphism between the sets of entities of type $X(c)$ and $X(d)$.

For example, the database of a veterinary clinic may contain entity types \(\mathrm{Dogs}\) and \(\mathrm{Humans}\), together with an ownership
morphism \(o\colon\mathrm{Dogs}\to\mathrm{Humans}\). An instance assigns
sets of registered dogs and humans, and the function
\(X(o)\colon X(\mathrm{Dogs})\to X(\mathrm{Humans})\) assigns to every
registered dog its registered owner.

The category of
elements of \(X\) determines a way of seeing a database as a category over the base category \(\CC\)
\[
\pi_X\colon\int^\CC X\longrightarrow\CC.
\]
where the functor \(\pi_X\) has the nice property of being a discrete opfibration.

A section of \(\pi_X\) as a functor is equivalent to a family of elements
\(x_c\in X(c)\), indexed by the objects of \(\CC\), such that
\[
X(f)(x_c)=x_d
\]
for every morphism \(f\colon c\to d\). Thus, a section represents a
globally coherent choice of records across the schema. More general
database queries and constraints have also been described categorically in
terms of lifting problems
\cite{SpivakDataMigration,SpivakDatabaseQueries}.

Moreover, for a database \(X\colon\CC\to\Sets\), global sections of
\(\pi_X\) are naturally identified with the elements of \(\lim_\CC X\).
Thus, the existence of a globally coherent selection is an instance of the
general problem of deciding whether the limit of a diagram is
non-empty. This problem has recently been studied from an algorithmic
perspective \cite{althaus2026parameterizedalgorithmtestinglimit} in the finite (in the sense of having finite diagrams and finite sets), and related existence problems for
compatible families arise naturally in combinatorics and constraint
satisfaction \cite{hadek2026categoricalperspectiveconstraintsatisfaction}. 

This leads to a quantitative refinement of the global section problem: when no globally coherent choice exists, we measure how far the database is from admitting one by determining the minimum number of subcategories covering the schema on which coherent choices can be made.

The aim of this article is to introduce categorical invariants that measure
this obstruction. The homotopical structure relevant to this problem is
intrinsically asymmetric. A natural transformation
\(\alpha\colon F\Rightarrow G\) has a prescribed orientation and need not
admit a transformation in the opposite direction. Rather than passing to
zigzags of natural transformations, as in the usual homotopy theory of
small categories \cite{LEEHomotopy}, we retain this orientation and regard \(\alpha\) as a directed homotopy from \(F\) to \(G\).

This point of view belongs to a broader tradition of using small
categories as combinatorial models for homotopical phenomena, which we
here develop in an explicitly directed setting.
The nerve and classifying-space constructions associate an ordinary
homotopy type with every small category
\cite{QuillenKTheory,ClassifyingSpace} and underlie model structures on
small categories \cite{Thomason,Cat_Miniam}.Passage to the ordinary homotopy type, however, does not
retain the full orientation of the morphisms. Related asymmetric phenomena
appear in directed algebraic topology and concurrency theory
\cite{DirectedConcurrency,GrandisDirected}, in the directed homotopy
hypothesis \cite{DirectedHomotopyHypothesis}, and in directed versions of
type theory
\cite{DirectedHomotopyHypothesis,DirectedTypeTheorySynthetic}.
Our approach is more elementary: we work with ordinary small categories,
natural transformations, and Grothendieck fibrations, without introducing
higher-categorical or foundational machinery.

Directed homotopy for small categories was developed by Grandis
\cite{GrandisShape}, who regarded a natural transformation as an oriented
homotopy and introduced future and past equivalences of categories. Building on this viewpoint, we consider one-sided notions of domination,
homotopy equivalence, and contractibility, closely related to his past and
future equivalences. We also use directed homotopies to define directed
fibrations and numerical covering invariants.

The asymmetry also produces two homotopy lifting properties. Depending on
whether the prescribed lift lies over the source or the target of a natural
transformation, we obtain right and left directed fibrations. These notions
are closely related to Grothendieck opfibrations and fibrations,
respectively \cite{GrothendieckRevetements}. We prove that every
Grothendieck opfibration is a right directed fibration and, dually, that
every Grothendieck fibration is a left directed fibration. We also construct
an opfibrational replacement: every functor factors as a right directed
homotopy equivalence followed by a Grothendieck opfibration.

Using this framework, we introduce two numerical invariants. The first is
the directed Lusternik--Schnirelmann category
\(\mathrm{dcat}(\CC)\), defined by covering \(\CC\) with subcategories
whose inclusions admit directed homotopies from constant functors. The
second is the directed sectional category \(\mathrm{dsecat}(P)\) of a
functor \(P\colon\E\to\Ba\), defined by requiring local right homotopy
sections. These invariants are one-sided analogues of categorical
Lusternik--Schnirelmann category and \v{S}varc genus previously studied in
the undirected setting
\cite{LS-Tanaka,Homotopic_Distance,Baues-Isaac,carcaciacampos2026weakstrongfibrationsfunctors}.

We establish their basic monotonicity and invariance properties. Directed
Lusternik--Schnirelmann category is invariant under right directed
homotopy equivalence, while directed sectional category is invariant under
compatible directed changes of the total and base categories. In
particular, every functor admits a Grothendieck opfibration model with the
same directed sectional category, on which local directed homotopy
sections can be strictified. We also obtain a canonical discrete lower
approximation from the comprehensive factorization \cite{StreetWalters}.

For a database \(X\colon\CC\to\Sets\), the associated projection
\(\pi_X\colon\int^\CC X\to\CC\) is a discrete opfibration. Hence
\(\mathrm{dsecat}(\pi_X)\) is computed by strict local sections and
measures the minimum normalized number of subcategories needed to cover
the schema so that coherent choices of records exist locally.

The article is organized as follows. Section~\ref{section:Directed_Homotpy} introduces directed
homotopy, domination, and contractibility. Section~\ref{sec:directed_fibrations} develops directed
fibrations, compares them with Grothendieck fibrations and opfibrations, and
constructs an opfibrational factorization. Section~\ref{sec:discrete_opfibrations_databases} describes discrete
opfibrations and their interpretation as functorial databases. Section~\ref{sec:categorical_covers}
introduces the morphism covers used in the sequel. Section~\ref{sec:directed_LS_category} introduces directed Lusternik--Schnirelmann category and studies
its invariance, bounds, and computations for acyclic categories, posets,
and monoids.  Section~\ref{sec:directed_sectional_category} develops directed sectional category, establishes its strictification and homotopy-invariance properties, compares its opfibrational and component models, and relates it to directed category and the undirected \v{S}varc genus. Section~\ref{sec:databases_sectional} applies the theory to indecomposable and iterated databases and to data migration, concluding with a characterization of initial objects through global sections.
\subsection*{Conventions}

Throughout the article, all categories are assumed to be small. Identity
morphisms and identity functors are denoted by \(1_c\) and \(1_\CC\),
respectively. Sometimes we will avoid writing the composition \(\circ\) to not saturate the reader.

A category \(\CC\) is \emph{acyclic} \cite[chapter 10]{Kozlov} if every endomorphism is an identity
and, for distinct objects \(c,d\in\CC\), the existence of a morphism
\(c\to d\) implies that there is no morphism \(d\to c\).

\section{Directed Homotopy in Small Categories}\label{section:Directed_Homotpy}

Homotopies between functors were introduced by Lee
\cite{LEEHomotopy} and subsequently studied in connection with categorical
models of homotopy theory \cite{Cat_Miniam}. The usual categorical notion
is the equivalence relation generated by natural transformations and hence
allows zigzags whose arrows may point in either direction. Consequently,
the orientation of the individual natural transformations is ultimately
forgotten. In this article, we retain this orientation and regard a natural
transformation itself as a directed homotopy. This point of view goes back
to Grandis \cite{GrandisShape}, who used the walking arrow category
\(\mathbb{I}_1=\{0\to1\}\) as a directed interval and identified directed
homotopies between functors with natural transformations.

\begin{definition}
Let \(F,G\colon\CC\to\Dc\) be functors. A \emph{directed homotopy} from
\(F\) to \(G\) is a natural transformation \(\alpha\colon F\Rightarrow G\).
We write \(F\leq_d G\) when such a natural transformation exists.
\end{definition}

This definition admits the expected cylinder description.  The following description is standard in directed
homotopy for categories; see \cite[Section~1.5]{GrandisShape}.

\begin{prop}\label{prop:directed_homotopy_cylinder}
Let \(F,G\colon\CC\to\Dc\) be functors. Then \(F\leq_d G\) if and only if
there exists a functor \(H\colon\CC\times\mathbb{I}_1\to\Dc\) such that
\(H(-,0)=F\) and \(H(-,1)=G\).
\end{prop}

\begin{proof}
A natural transformation \(\alpha\colon F\Rightarrow G\) determines
\(H\) by setting \(H(c,0)=F(c)\), \(H(c,1)=G(c)\), and
\(H(1_c,s_0)=\alpha_c\); its values on the two copies of \(\CC\) are
prescribed by \(F\) and \(G\). The naturality of \(\alpha\) is precisely the
condition required for \(H\) to be functorial.

Conversely, if such a functor \(H\) is given, the morphisms
\(\alpha_c=H(1_c,s_0) \colon \allowbreak F(c)\to G(c)\) form a natural transformation
\(\alpha\colon F\Rightarrow G\).
\end{proof}
The functor category \([\CC,\Dc]\) induces a preorder on its objects by
declaring
\[
F\leq_d F'
\quad\Longleftrightarrow\quad
\operatorname{Nat}(F,F')\neq\varnothing.
\]
Reflexivity and transitivity follow from identity transformations and
vertical composition, respectively. Horizontal composition shows that
this preorder is compatible with composition of functors: if
\(F\leq_d F'\), then
\[
F \circ G\leq_d F'\circ G
\qquad\text{and}\qquad
K\circ F\leq_d K \circ F'
\]
whenever the corresponding composites are defined. More generally, if
\(F\leq_d F'\) and \(G\leq_d G'\), then
\[
GF\leq_d G'F'
\]
whenever these composites are defined. Thus, the hom-categories of
\(\mathbf{Cat}\) become preordered by directed homotopy, and composition
is monotone in both variables.

In general, this preorder need not be symmetric or antisymmetric. If
\(\Dc\) is acyclic, however, then \([\CC,\Dc]\) is also acyclic, and the
induced preorder is a partial order. We next introduce one-sided versions
of domination and homotopy equivalence.

\begin{definition}\label{def:directed_domination}
Let \(\CC\) and \(\Dc\) be categories.

We say that \(\CC\) is \emph{right dominated} by \(\Dc\), and write \(\CC\unlhd_R\Dc\), if there exist functors \(F\colon\CC\to\Dc\) and \(G\colon\Dc\to\CC\) such that \(G\circ F\leq_d1_\CC\).

Dually, \(\CC\) is \emph{left dominated} by \(\Dc\), written
\(\CC\unlhd_L\Dc\), if there exist such functors satisfying
\(1_\CC\leq_dG\circ F\).

The categories \(\CC\) and \(\Dc\) are \emph{right directed homotopy
equivalent}, written \(\CC\simeq_R\Dc\), if
\(\CC\unlhd_R\Dc\) and \(\Dc\unlhd_R\CC\). Left directed homotopy
equivalence, denoted by \(\CC\simeq_L\Dc\), is defined dually.
\end{definition}

\begin{remark}
Our right and left directed homotopy equivalences are related,
respectively, to the past and future equivalences introduced by Grandis
\cite[Section~2]{GrandisShape}. More precisely, the orientations of the
natural transformations agree, but Grandis additionally imposes coherence
identities relating the two transformations. No such coherence conditions
are required in the present article.
\end{remark}

\subsection{Contractibility}\label{subsec:cont}

Past and future contractibility were studied by Grandis, who related them
to initial and terminal objects, respectively
\cite[Section~2.6]{GrandisShape}. We now consider the corresponding
one-sided notions without imposing the coherence conditions of a past or
future equivalence. Let \(\bullet\) denote the category with one object and one morphism. A
functor \(\CC\to\bullet\) is unique, whereas a functor
\(\bullet\to\CC\) amounts to choosing an object of \(\CC\). Directed
homotopy equivalence with \(\bullet\) therefore leads to the following
notions.

\begin{definition}\label{def:homotopically_initial_terminal}
An object \(i\in\CC\) is \emph{homotopically initial} if there exists a
natural transformation \(\bar{i}\Rightarrow1_\CC\). Dually, an object
\(t\in\CC\) is \emph{homotopically terminal} if there exists a natural
transformation \(1_\CC\Rightarrow\bar{t}\).
\end{definition}

Thus, \(\CC\) is right directed homotopy equivalent to \(\bullet\) if and
only if it has a homotopically initial object. The corresponding left-handed
statement is dual.

\begin{prop}\label{prop:acyclic_directed_contractibility}
Let \(\CC\) be an acyclic category. Then \(\CC\simeq_R\bullet\) if and
only if \(\CC\) has an initial object. Dually,
\(\CC\simeq_L\bullet\) if and only if \(\CC\) has a terminal object.
\end{prop}

\begin{proof}
We prove the right-handed statement. Suppose that
\(\alpha\colon\bar{i}\Rightarrow1_\CC\). For every object \(c\), the
component \(\alpha_c\colon i\to c\) provides a morphism from \(i\) to \(c\).
Since \(\CC\) is acyclic, \(\alpha_i=1_i\). If \(f\colon i\to c\) is any morphism, naturality with respect to \(f\) yields
\(f\circ\alpha_i=\alpha_c\), and hence \(f=\alpha_c\). Therefore \(i\) is
initial.

Conversely, if \(i\) is initial, the unique morphisms
\(\alpha_c\colon i\to c\) form a natural transformation
\(\bar{i}\Rightarrow1_\CC\). Thus \(\CC\simeq_R\bullet\).
\end{proof}

The acyclicity hypothesis is essential. Let \(M=\{1,s\}\) be the monoid
defined by \(s^2=s\). Regarded as a one-object category, \(M\) admits a
natural transformation \(\bar{\ast}\Rightarrow1_M\) with component \(s\),
because \(m\circ s=s\) for every \(m\in M\). Nevertheless, its unique object
is not initial, since it has two endomorphisms. More generally, a monoid is
right directed contractible precisely when it has a right absorbing element;
the left-handed statement is dual.

\begin{remark}
The terminology of right and left agrees with the convention for absorbing elements. If a monoid \(M\) is regarded as a one-object category, a natural transformation \(\bar{\ast}\Rightarrow1_M\) is determined by an element \(s\in M\) satisfying \(m\circ s=s\) for every \(m\in M\), that is, by a right absorbing element. The dual transformation \(1_M\Rightarrow\bar{\ast}\) corresponds to a left absorbing element.
\end{remark}

\section{Directed Fibrations in Small Categories}\label{sec:directed_fibrations}

The orientation of a directed homotopy distinguishes its two endpoints. A
natural transformation \(F'\Rightarrow G'\) cannot in general be reversed,
so prescribing a lift over \(F'\) and prescribing one over \(G'\) lead to
different lifting problems. We therefore introduce right and left directed
fibrations.

We shall compare these notions with Grothendieck fibrations and
opfibrations
\cite{RosoliniGrothendieckFibrations,RiehlLoregianFibration,Varadarajan}.

\subsection{Directed fibrations}

For a category \(\CC\) and \(\varepsilon\in\{0,1\}\), let
\(\iota_\varepsilon^\CC\colon\CC\to\CC\times\mathbb{I}_1\) be the functor
defined by \(\iota_\varepsilon^\CC(c)=(c,\varepsilon)\). We omit the
superscript whenever the category is clear.

\begin{definition}\label{def:directed_fibration}
Let \(P\colon\E\to\Ba\) be a functor.

We say that \(P\) is a \emph{right directed fibration} if, for every small
category \(\CC\) and every commutative square
\[
\begin{tikzcd}
\CC \arrow[r, "F"] \arrow[d, "\iota_0"'] &
\E \arrow[d, "P"] \\
\CC\times\mathbb{I}_1 \arrow[r, "H"'] &
\Ba ,
\end{tikzcd}
\]
there exists a functor
\(\widetilde H\colon\CC\times\mathbb{I}_1\to\E\) such that
\(P\widetilde H=H\) and \(\widetilde H\iota_0=F\).

Dually, \(P\) is a \emph{left directed fibration} if it has the analogous
lifting property with respect to \(\iota_1\).
\end{definition}

\begin{remark}\label{Rem:Lift_Natural}
The right lifting property can equivalently be expressed in terms of
natural transformations. Suppose that
\(F',G'\colon\CC\to\Ba\), that
\(\alpha'\colon F'\Rightarrow G'\), and that
\(F\colon\CC\to\E\) satisfies \(PF=F'\). Then there must exist a functor
\(G\colon\CC\to\E\) and a natural transformation
\(\alpha\colon F\Rightarrow G\) such that
\(PG=G'\) and \(P(\alpha_c)=\alpha'_c\) for every \(c\in\CC\).

For the left lifting property, one starts instead with a lift \(G\) of
\(G'\) and asks for a lift \(F\) of \(F'\), together with a natural
transformation \(\alpha\colon F\Rightarrow G\) lifting \(\alpha'\).
\end{remark}

\subsection{Grothendieck fibrations and opfibrations}

Grothendieck opfibrations lift morphisms covariantly, whereas Grothendieck
fibrations provide the corresponding contravariant lifting property. We
recall only the opfibrational definitions; the fibrational versions are
dual. See
\cite{RosoliniGrothendieckFibrations,RiehlLoregianFibration,Varadarajan}
for further details.

\begin{definition}
Let \(P\colon\E\to\Ba\) be a functor. A morphism
\(\varphi\colon e_1\to e_2\) in \(\E\) is \emph{\(P\)-opcartesian} if,
for every morphism \(\beta\colon e_1\to e\) in \(\E\) and every morphism
\(\overline{\alpha}\colon P(e_2)\to P(e)\) in \(\Ba\) satisfying
\(\overline{\alpha}\circ P(\varphi)=P(\beta)\), there exists a unique
morphism \(\alpha\colon e_2\to e\) such that
\(\alpha\circ\varphi=\beta\) and
\(P(\alpha)=\overline{\alpha}\).
Equivalently, every diagram of the following form admits a unique dashed
factorization.
			\[\begin{tikzcd}
				e_1 \ar[d,"\varphi"']\arrow[r,"\beta"] &e \\
				e_2 \arrow[ur,dashrightarrow,"\alpha"']&  
			\end{tikzcd}
			\quad \quad
			\begin{tikzcd}
				Pe_1 \arrow[d,"P\varphi"'] \arrow[r,"P\beta"] &Pe \\
				Pe_2\arrow[ur,"{\bbar\alpha}"']& 
			\end{tikzcd}
			\]
			
\end{definition}

\begin{definition}
A functor \(P\colon\E\to\Ba\) is a \emph{Grothendieck opfibration} if,
for every morphism \(\overline{\varphi}\colon b_1\to b_2\) in \(\Ba\)
and every \(e_1\in\E\) with \(P(e_1)=b_1\), there exists a
\(P\)-opcartesian morphism \(\varphi\colon e_1\to e_2\) satisfying
\(P(\varphi)=\overline{\varphi}\).
\end{definition}

\begin{remark}
The dual notions of \(P\)-cartesian morphism and Grothendieck fibration
are obtained by reversing all arrows. Equivalently,
\(P\colon\E\to\Ba\) is a Grothendieck fibration if and only if
\(P^{\op}\colon\E^{\op}\to\Ba^{\op}\) is a Grothendieck opfibration.
We shall state most results in their opfibrational form; the corresponding
fibrational statements follow by duality. See
\cite{RosoliniGrothendieckFibrations,GrothendieckRevetements,RiehlLoregianFibration,Varadarajan} for further background.
\end{remark}

The Grothendieck construction gives a systematic method for producing
opfibrations from covariantly indexed categories.

\begin{definition}
Let \(P\colon\Ba\to\cat\) be a functor. Its \emph{Grothendieck
construction}, denoted by \(\int^\Ba P\), is the category whose objects are
pairs \((b,x)\), where \(b\in\Ba\) and \(x\in P(b)\). A morphism
\[
(\varphi,f)\colon(b,x)\longrightarrow(b',y)
\]
consists of a morphism \(\varphi\colon b\to b'\) in \(\Ba\) and a
morphism \(f\colon P(\varphi)(x)\to y\) in \(P(b')\).

If
\((\psi,g)\colon(b',y)\to(b'',z)\), their composite is
\[
(\psi,g)\circ(\varphi,f)
=
\bigl(\psi\circ\varphi,\,
g\circ P(\psi)(f)\bigr).
\]
The projection
\(\pi_P\colon\int^\Ba P\to\Ba\) is defined by
\(\pi_P(b,x)=b\) and \(\pi_P(\varphi,f)=\varphi\).
\end{definition}

\begin{prop}\label{prop:GrothendiecktConstruction_to_Opfibration}
Let \(P\colon\Ba\to\cat\) be a functor. Then
\(\pi_P\colon\int^\Ba P\to\Ba\) is a split Grothendieck opfibration.
Dually, a functor \(P\colon\Ba^{\op}\to\cat\) determines a split
Grothendieck fibration over \(\Ba\).
\end{prop}

\begin{proof}
Let \(\varphi\colon b_1\to b_2\) be a morphism in \(\Ba\), and let
\((b_1,x)\) lie over \(b_1\). Consider
\[
(\varphi,1_{P(\varphi)(x)})\colon
(b_1,x)\longrightarrow(b_2,P(\varphi)(x)).
\]
We show that this morphism is opcartesian.

Suppose that
\((\beta,g)\colon(b_1,x)\to(b,y)\) and that
\(\overline{\alpha}\colon b_2\to b\) satisfy
\(\overline{\alpha}\circ\varphi=\beta\). Then
\[
g\colon
P(\beta)(x)
=
P(\overline{\alpha})(P(\varphi)(x))
\longrightarrow y,
\]
so \((\overline{\alpha},g)\) is a morphism from
\((b_2,P(\varphi)(x))\) to \((b,y)\). It is the unique morphism over
\(\overline{\alpha}\) whose composite with
\((\varphi,1_{P(\varphi)(x)})\) is \((\beta,g)\). Hence the latter is
opcartesian.

These chosen lifts preserve identities and composition, so
\(\pi_P\) is split.
\end{proof}

\begin{prop}\label{prop:Grothendieck_implies_directed}
Every Grothendieck opfibration is a right directed fibration. Dually, every
Grothendieck fibration is a left directed fibration.

Conversely, a right directed fibration is a Grothendieck opfibration if its
liftings can always be chosen so that every component of the lifted natural
transformation is opcartesian.
\end{prop}

\begin{proof}
Let \(P\colon\E\to\Ba\) be a Grothendieck opfibration. Consider functors
\(F',G'\colon\CC\to\Ba\), a natural transformation
\(\alpha'\colon F'\Rightarrow G'\), and a lift
\(F\colon\CC\to\E\) of \(F'\).

Choose, for every \(c\in\CC\), an opcartesian lift
\(\alpha_c\colon F(c)\to G(c)\) of
\(\alpha'_c\colon F'(c)\to G'(c)\). For a morphism
\(f\colon c\to c'\), naturality of \(\alpha'\) gives
\[
G'(f)\circ\alpha'_c
=
\alpha'_{c'}\circ F'(f).
\]
Since \(\alpha_c\) is opcartesian, there is a unique morphism
\(G(f)\colon G(c)\to G(c')\) over \(G'(f)\) such that
\[
G(f)\circ\alpha_c
=
\alpha_{c'}\circ F(f).
\]
Uniqueness shows that \(G\) preserves identities and composition. Thus
\(G\colon\CC\to\E\) is a functor and the family
\((\alpha_c)_{c\in\CC}\) defines a natural transformation
\(\alpha\colon F\Rightarrow G\) lifting \(\alpha'\). Hence \(P\) is a
right directed fibration.

For the converse, apply the lifting property with
\(\CC=\bullet\). A morphism in \(\Ba\), together with an object of \(\E\)
over its source, determines a natural transformation indexed by
\(\bullet\). By hypothesis, its lift may be chosen opcartesian, which is
precisely the lifting condition for a Grothendieck opfibration.
\end{proof}

\begin{remark}
The proof uses a simultaneous choice of opcartesian lifts. This is
automatic for cloven opfibrations and canonical for split opfibrations.
For arbitrary small opfibrations, it uses the axiom of choice.
\end{remark}

\subsection{Factorization by opfibrations}\label{subsec:factorization_pfibration}

Comma categories play the role of directed homotopy pullbacks in the
framework of Grandis, where they are also used to factor adjunctions and
construct directed models
\cite[Sections~1.6 and~4.4]{GrandisShape}. We use a related comma-category
construction to obtain an opfibrational factorization of an arbitrary
functor. More precisely, we will show that every functor factors as a right directed homotopy equivalence followed by a Grothendieck opfibration.

\begin{proposition}\label{prop:directed_opfibration_factorization}
Every functor \(F\colon\CC\to\Dc\) factors as
\[
\CC\xrightarrow{F_1}E_F\xrightarrow{F_2}\Dc,
\]
where \(F_1\) is a right directed homotopy equivalence and \(F_2\) is a
Grothendieck opfibration.
\end{proposition}

\begin{proof}
Let \(E_F=F\downarrow1_\Dc\). We write its objects as morphisms
\(h\colon F(c)\to d\), and its morphisms as pairs
\[
(g,u)\colon h\longrightarrow h'
\]
satisfying \(uh=h'F(g)\).

Define
\[
F_1(c)=1_{F(c)},
\qquad
F_1(g)=(g,F(g)),
\]
and let \(R\colon E_F\to\CC\) send \(h\colon F(c)\to d\) to \(c\) and
\((g,u)\) to \(g\). Then \(RF_1=1_\CC\), while the morphisms
\[
(1_c,h)\colon F_1R(h)\longrightarrow h
\]
form a natural transformation \(F_1R\Rightarrow1_{E_F}\). Hence \(F_1\)
is a right directed homotopy equivalence.

Define \(F_2(h\colon F(c)\to d)=d\) and \(F_2(g,u)=u\). Then
\(F_2F_1=F\). Given \(u\colon d\to d'\) and an object
\(h\colon F(c)\to d\), the morphism
\[
(1_c,u)\colon h\longrightarrow uh
\]
lies over \(u\). If \((g,v)\colon h\to h'\) lies over \(v=wu\), then
\(h'F(g)=vh=wuh\), so there is a unique morphism
\[
(g,w)\colon uh\longrightarrow h'
\]
over \(w\) whose composite with \((1_c,u)\) is \((g,v)\). Thus
\((1_c,u)\) is opcartesian, and \(F_2\) is a Grothendieck opfibration.
\end{proof}

\begin{remark}
The opfibration \(F_2\) is the Grothendieck construction of the functor
\[
F\downarrow(-)\colon\Dc\longrightarrow\cat,
\qquad
d\longmapsto F\downarrow d.
\]
A morphism \(u\colon d\to d'\) induces a functor
\(F\downarrow d\to F\downarrow d'\) by postcomposition,
\(h\mapsto u\circ h\).
\end{remark}

\begin{remark}
The construction admits a dual version. Using the comma category
\(1_\Dc\downarrow F\), every functor \(F\colon\CC\to\Dc\) factors as a
left directed homotopy equivalence followed by a Grothendieck fibration.
We emphasize the opfibrational factorization because a functor
\(X\colon\CC\to\Sets\), interpreted as a database, determines a discrete
opfibration through its Grothendieck construction.
\end{remark}
\section{Discrete Opfibrations and Databases}
\label{sec:discrete_opfibrations_databases}

A Grothendieck opfibration may be regarded as a family of categories indexed
by the objects of its base. When the fibres are discrete, this family is
set-valued and the lifting property becomes an existence-and-uniqueness
condition.

\begin{definition}\label{def:discrete_opfibration}
A functor \(P\colon\E\to\Ba\) is a \emph{discrete opfibration} if, for every
morphism \(f\colon b\to b'\) in \(\Ba\) and every \(e\in\E\) with
\(P(e)=b\), there exists a unique morphism
\(\overline f\colon e\to e'\) satisfying \(P(\overline f)=f\).
\end{definition}

Thus, both the lift \(\overline f\) and its target are uniquely determined
by \(f\) and \(e\).

\begin{prop}\label{prop:unique_lifts_are_opcartesian}
Every morphism of a discrete opfibration \(P\colon\E\to\Ba\) is
\(P\)-opcartesian. In particular, every discrete opfibration is a
Grothendieck opfibration.
\end{prop}

\begin{proof}
Let \(\varphi\colon e_1\to e_2\), and suppose that
\(\beta\colon e_1\to e\) and
\(\overline\alpha\colon P(e_2)\to P(e)\) satisfy
\(\overline\alpha P(\varphi)=P(\beta)\). Let
\(\alpha\colon e_2\to e'\) be the unique lift of
\(\overline\alpha\) with source \(e_2\).

The morphisms \(\alpha\varphi\) and \(\beta\) have the same source and lie
over the same morphism. Uniqueness of lifts gives
\(\alpha\varphi=\beta\) and \(e'=e\). The required factorization is
therefore unique, so \(\varphi\) is opcartesian.
\end{proof}

Conversely, an opfibration is discrete precisely when its fibres are
discrete.

\begin{prop}\label{prop:discrete_fibres_discrete_opfibration}
Let \(P\colon\E\to\Ba\) be a Grothendieck opfibration. If every fibre
\(\E_b\) is discrete, then \(P\) is a discrete opfibration.
\end{prop}

\begin{proof}
Let \(\overline f\colon e\to e'\) be an opcartesian lift of
\(f\colon b\to b'\). If \(\gamma\colon e\to e''\) is another lift, the
opcartesian property gives a unique morphism
\(\alpha\colon e'\to e''\) over \(1_{b'}\) such that
\(\alpha\overline f=\gamma\). Since \(\alpha\) belongs to the discrete
fibre \(\E_{b'}\), it is an identity. Hence
\(e'=e''\) and \(\gamma=\overline f\).
\end{proof}

\subsection{The Category of Elements}

Let \(X\colon\Ba\to\Sets\). Its \emph{category of elements}
\(\int^\Ba X\) has objects \((b,x)\), with \(x\in X(b)\), and a morphism
\((b,x)\to(b',x')\) is a morphism \(f\colon b\to b'\) satisfying
\(X(f)(x)=x'\). The projection
\[
\pi_X\colon\int^\Ba X\longrightarrow\Ba
\]
sends \((b,x)\) to \(b\) and every morphism over \(f\) to \(f\).

The Grothendieck construction induces the standard equivalence
\[
[\Ba,\Sets]\simeq\operatorname{DOpFib}(\Ba)
\]
between set-valued functors and discrete opfibrations over \(\Ba\); see
\cite{RiehlLoregianFibration,SpivakDataMigration}. Under this equivalence,
a discrete opfibration \(P\colon\E\to\Ba\) corresponds to the functor
\[
b\longmapsto\operatorname{Obj}(\E_b),
\]
whose action on \(f\colon b\to b'\) sends \(e\in\E_b\) to the target
\(f_!e\) of the unique lift of \(f\) with source \(e\).

\subsection{Functorial Databases}

In the functorial model of databases, a schema is a small category \(\Ba\)
and an instance on that schema is a functor
\(X\colon\Ba\to\Sets\)
\cite{SevenSketches,SpivakDataMigration,SpivakDatabaseQueries}. An object
\(b\in\Ba\) represents an entity type, \(X(b)\) is its set of records, and
a morphism \(f\colon b\to b'\) represents a functional attribute or foreign
key. The function \(X(f)\colon X(b)\to X(b')\) assigns to each record its
corresponding value under that attribute.

\begin{example}\label{example:animals_database}
Consider the schema \(\Ba_{\mathrm{DCH}}\) generated by
\[
\begin{tikzcd}
\mathrm{Dogs} \arrow[rrd, "o_D"] && \\
\mathrm{Cats} \arrow[rr, "o_C"'] && \mathrm{Humans}.
\end{tikzcd}
\]
The morphisms \(o_D\) and \(o_C\) represent ownership. An instance assigns
sets of registered dogs, cats, and humans, together with functions assigning
to every registered animal its unique registered owner.

For example, let
\[
\begin{aligned}
X(\mathrm{Dogs})  &= \{d_1,d_2,d_3\}, \\
X(\mathrm{Cats})  &= \{c_1,c_2\}, \\
X(\mathrm{Humans}) &= \{h_1,h_2,h_3\},
\end{aligned}
\]
with
\[
\begin{aligned}
X(o_D)(d_1) &= X(o_D)(d_2)=h_1, \\
X(o_D)(d_3) &= X(o_C)(c_1)=h_2, \\
X(o_C)(c_2) &= h_3.
\end{aligned}
\]
The corresponding category of elements is represented by
\[
\begin{tikzcd}[column sep=2em, row sep=0.8em]
(\mathrm{Dog},d_1) \arrow[rd] & \\
(\mathrm{Dog},d_2) \arrow[r]  & (\mathrm{Human},h_1) \\
(\mathrm{Dog},d_3) \arrow[r]  & (\mathrm{Human},h_2) \\
(\mathrm{Cat},c_1) \arrow[ru] & \\
(\mathrm{Cat},c_2) \arrow[rd] & \\
                              & (\mathrm{Human},h_3).
\end{tikzcd}
\]
The projection \(\pi_X\) sends every object to its entity type and every
morphism to the corresponding ownership morphism.
\end{example}
\subsection{Sections and Globally Coherent Data}

\begin{prop}\label{prop:sections_database_coherent_families}
Let \(X\colon\Ba\to\Sets\). A section of
\(\pi_X\colon\int^\Ba X\to\Ba\) is equivalent to a family
\((x_b)_{b\in\Ba}\), with \(x_b\in X(b)\), satisfying
\[
X(f)(x_b)=x_{b'}
\]
for every morphism \(f\colon b\to b'\).
\end{prop}

\begin{proof}
If \(s\) is a section, write \(s(b)=(b,x_b)\). For
\(f\colon b\to b'\), the morphism \(s(f)\) exists in the category of
elements precisely when \(X(f)(x_b)=x_{b'}\).

Conversely, a family satisfying these equations defines a section by
\(s(b)=(b,x_b)\) and \(s(f)=f\).
\end{proof}

Thus, a section is a globally coherent choice of one record of every entity
type: the selected records agree with every attribute and foreign-key map.
Equivalently, a section determines an element of the limit of \(X\). When
such a global choice does not exist, one may ask whether coherent choices
exist after restricting the schema to suitable subcategories. This will be the key idea behind Section~\ref{sec:directed_sectional_category}.

\begin{example}\label{example:monoid_action_database}
Let \(M\) be a monoid regarded as a one-object category. A functor
\(X\colon M\to\Sets\) is a left action of \(M\) on \(X(\ast)\).

The category of elements \(\int^M X\) is the corresponding action category:
its objects are the elements \(x\in X(\ast)\), and a morphism labelled by
\(m\in M\) from \(x\) to \(y\) exists precisely when \(m\cdot x=y\).

A section of \(\pi_X\colon\int^M X\to M\) is therefore an element
\(x\in X(\ast)\) satisfying \(m\cdot x=x\) for every \(m\in M\). Thus
global sections are precisely the common fixed points of the action.
\end{example}

\subsection{The comprehensive factorization}
\label{subsec:discrete_factorization}

The comprehensive factorization of Street and Walters associates a
canonical database with every functor \cite{StreetWalters}. It will later
provide a discrete lower approximation to directed sectional category.

Let \(F\colon\CC\to\Dc\). Define its \emph{component database}
\(K_F\colon\Dc\to\Sets\) by
\[
K_F(d)=\pi_0(F\downarrow d),
\]
where \(\pi_0\) denotes connected components with respect to unoriented
zigzags. A morphism \(u\colon d\to d'\) acts by postcomposition:
\[
K_F(u)[c,\alpha]=[c,u\alpha].
\]

Let
\[
q_F\colon\int^\Dc K_F\longrightarrow\Dc
\]
be the associated discrete opfibration. There is a canonical functor
\[
i_F\colon\CC\longrightarrow\int^\Dc K_F,
\qquad
i_F(c)=\bigl(F(c),[(c,1_{F(c)})]\bigr),
\]
and \(F=q_Fi_F\).

\begin{proposition}[Street--Walters]
\label{prop:comprehensive_factorization}
Every functor \(F\colon\CC\to\Dc\) admits a functorial factorization
\[
\CC
\xrightarrow{i_F}
\int^\Dc K_F
\xrightarrow{q_F}
\Dc,
\]
where \(i_F\) is initial and \(q_F\) is a discrete opfibration.
\end{proposition}

\begin{proof}
The functor \(q_F\) is a discrete opfibration because it is a
category-of-elements projection. For
\((d,\xi)\in\int^\Dc K_F\), the comma category
\(i_F\downarrow(d,\xi)\) is naturally isomorphic to the full subcategory
of \(F\downarrow d\) determined by the component \(\xi\). It is therefore
non-empty and connected, so \(i_F\) is initial.
\end{proof}

For Grothendieck opfibrations, the comprehensive factorization amounts to
taking connected components fibrewise.

\begin{proposition}\label{prop:comprehensive_grothendieck}
Let \(X\colon\CC\to\cat\), and let
\[
P_X\colon\int^\CC X\longrightarrow\CC
\]
be its Grothendieck opfibration. Then
\[
K_{P_X}\cong\pi_0X,
\qquad
(\pi_0X)(c)=\pi_0(X(c)),
\]
and the comprehensive factorization of \(P_X\) is naturally isomorphic to
\[
\int^\CC X
\xrightarrow{\kappa_X}
\int^\CC\pi_0X
\xrightarrow{\pi_{\pi_0X}}
\CC,
\]
where \(\kappa_X(c,x)=(c,[x])\).
\end{proposition}

\begin{proof}
For every \(d\in\CC\), an object of \(P_X\downarrow d\) has the form
\[
\bigl((c,x),f\colon c\to d\bigr).
\]
The opcartesian lift of \(f\) connects this object to
\[
\bigl((d,X(f)(x)),1_d\bigr).
\]
Thus every component of \(P_X\downarrow d\) contains an object of the
fibre \(X(d)\), and two such objects belong to the same component precisely
when they are connected in \(X(d)\). Hence
\[
\pi_0(P_X\downarrow d)\cong\pi_0(X(d)),
\]
naturally in \(d\), and therefore \(K_{P_X}\cong\pi_0X\).
\end{proof}

\begin{remark}
Thus, the comprehensive factorization of a Grothendieck opfibration
replaces each categorical fibre by its set of connected components. In
particular, if \(P_X\) is a discrete opfibration, then every \(X(d)\) is
discrete, so \(\pi_0X\cong X\) and
\[
\pi_{\pi_0X}\cong P_X
\]
over \(\CC\).
\end{remark}
\section{Covers of Small Categories}
\label{sec:categorical_covers}

The invariants introduced below are defined by decomposing a category into
subcategories on which a prescribed homotopical or sectional condition
holds. We first specify the notion of cover used throughout the article.

\begin{definition}\label{def:categorical_cover}
A \emph{morphism cover} of a category \(\CC\) is a family of subcategories
\(\{\Uc_i\}_{i\in I}\) such that every morphism of \(\CC\) belongs to at
least one \(\Uc_i\).
\end{definition}

Since every identity morphism \(1_c\) belongs to some member of the cover,
each object \(c\in\CC\) also belongs to at least one \(\Uc_i\). Throughout
the remainder of the article, \emph{cover} will always mean a morphism
cover in the sense of Definition~\ref{def:categorical_cover}.

\begin{remark}\label{rem:morphism_vs_geometric_cover}
Morphism covers are weaker than the geometric covers commonly used in
categorical versions of Lusternik--Schnirelmann category, sectional
category, and homotopic distance
\cite{LS-Tanaka,Homotopic_Distance,Baues-Isaac,
carcaciacampos2026weakstrongfibrationsfunctors}.
A geometric cover requires every finite composable sequence of morphisms
to be contained in one member of the cover.

This stronger requirement is natural when the categorical invariant is
intended to model an invariant of the nerve or classifying space, since
the simplices of the nerve are precisely composable sequences of
morphisms. Our purpose is different: we retain the directed categorical
structure itself and use subcategories as local domains on which directed
deformations or coherent selections may be defined. In particular, for a
database \(X\colon\CC\to\Sets\), the members of a cover represent regions
of the schema on which compatible choices of records can be made.

Morphism covers are also preserved by inverse images, a property used
repeatedly below.
\end{remark}

\begin{lemma}\label{lem:inverse_image_cover}
Let \(F\colon\CC\to\Dc\) be a functor. If
\(\{\Uc_i\}_{i\in I}\) is a cover of \(\Dc\), then
\(\{F^{-1}(\Uc_i)\}_{i\in I}\) is a cover of \(\CC\).
\end{lemma}

\begin{proof}
For every morphism \(f\) of \(\CC\), the morphism \(F(f)\) belongs to some
\(\Uc_i\). Hence \(f\in F^{-1}(\Uc_i)\).
\end{proof}

We use the normalized convention for all sectional invariants: an invariant
equal to \(n\) corresponds to a cover by \(n+1\) subcategories.
\section{Directed Lusternik--Schnirelmann Category}
\label{sec:directed_LS_category}

The Lusternik--Schnirelmann category is a classical invariant measuring the
minimum number of open subsets that are null-homotopic in the ambient
space; see \cite{LS-CAT-OPREA}. A directed Lusternik--Schnirelmann category for directed topological spaces
was recently introduced by Datta, Daundkar, and Sarkar
\cite{DattaDaundkarSarkar}. Categorical analogues for the usual, non-directed, LS-category have been defined
using zigzags of natural transformations
\cite{LS-Tanaka,carcaciacampos2026weakstrongfibrationsfunctors}. The invariant defined below is of a different
nature: it is formulated directly for small categories, using morphism
covers and oriented natural transformations.

\begin{definition}\label{def:directed_LS_category}
The \emph{directed Lusternik--Schnirelmann category} of a small category
\(\CC\), denoted by \(\mathrm{dcat}(\CC)\), is the least integer \(n\geq0\)
for which \(\CC\) admits a cover
\(\{\Uc_0,\ldots,\Uc_n\}\) such that, for every \(i\), there exist an
object \(c_i\in\CC\) and a natural transformation
\[
\alpha^i\colon\overline{c_i}\Rightarrow\iota_i,
\]
where \(\overline{c_i}\colon\Uc_i\to\CC\) is the constant functor with
value \(c_i\) and \(\iota_i\colon\Uc_i\hookrightarrow\CC\) is the
inclusion.

Such a subcategory is called \emph{right categorical}, and a cover by right
categorical subcategories is called a \emph{right categorical cover}. If
no finite right categorical cover exists, we set
\(\mathrm{dcat}(\CC)=\infty\).
\end{definition}

Thus, \(\mathrm{dcat}(\CC)\) measures the minimum number of local pieces on
which a constant functor can be deformed towards the corresponding
inclusion and the value zero recovers directed contractibility.

\begin{proposition}\label{prop:dcat_zero}
Let \(\CC\) be a non-empty small category. Then
\[
\mathrm{dcat}(\CC)=0
\]
if and only if \(\CC\) has a homotopically initial object. If, moreover,
\(\CC\) is acyclic, these conditions are equivalent to the existence of an
initial object.
\end{proposition}

\begin{proof}
The equality \(\mathrm{dcat}(\CC)=0\) means precisely that the cover
\(\{\CC\}\) is right categorical, that is, that there exist an object
\(i\in\CC\) and a natural transformation
\(\overline{i}\Rightarrow1_\CC\). Thus \(i\) is homotopically initial.
The statement for acyclic categories follows from
Proposition~\ref{prop:acyclic_directed_contractibility}.
\end{proof}
\subsection{Domination and Directed Homotopy Equivalence}

\begin{prop}\label{prop:dcat_domination}
Let \(\CC\) and \(\Dc\) be small categories. If
\(\CC\unlhd_R\Dc\), then
\(\mathrm{dcat}(\CC)\leq\mathrm{dcat}(\Dc)\). Consequently, if
\(\CC\simeq_R\Dc\), then
\(\mathrm{dcat}(\CC)=\mathrm{dcat}(\Dc)\).
\end{prop}

\begin{proof}
Choose \(F\colon\CC\to\Dc\), \(G\colon\Dc\to\CC\), and
\(\gamma\colon GF\Rightarrow1_\CC\).

The assertion is immediate if \(\mathrm{dcat}(\Dc)=\infty\). Otherwise,
let \(\{\Uc_0,\ldots,\Uc_n\}\) be a minimal right categorical cover, with
\(\overline{d_i}\leq_d\iota_i\). Set
\(\Vc_i=F^{-1}(\Uc_i)\), and denote by
\(\iota_i'\colon\Vc_i\hookrightarrow\CC\) the inclusion. Applying \(G\),
restricting along \(F|_{\Vc_i}\), and then using \(\gamma\), gives
\[
\overline{G(d_i)}
\leq_d
GF\iota_i'
\leq_d
\iota_i'.
\]
Thus every \(\Vc_i\) is right categorical. Since the \(\Vc_i\) cover
\(\CC\), we obtain
\[
\mathrm{dcat}(\CC)\leq\mathrm{dcat}(\Dc).
\]
The equality for right directed homotopy equivalent categories follows by
applying the inequality in both directions.
\end{proof}

\subsection{Bounds for Finite Acyclic Categories}

\begin{definition}
Let \(\CC\) be an acyclic category. An object \(s\in\CC\) is a
\emph{source} if \(\Hom_\CC(x,s)=\varnothing\) for every \(x\neq s\).
\end{definition}

Thus, a source receives no non-identity morphisms.

\begin{remark}
Every finite non-empty acyclic category has a source. Indeed, starting at
any object and repeatedly following a non-identity incoming morphism must
terminate: acyclicity prevents repetition and finiteness prevents an
infinite sequence.
\end{remark}

\begin{prop}\label{prop:dcat_sources_lower_bound}
Let \(\CC\) be a finite non-empty acyclic category. Then
\[
\mathrm{dcat}(\CC)
\geq
|\operatorname{Src}(\CC)|-1.
\]
\end{prop}

\begin{proof}
Suppose that \(\{\Uc_0,\ldots,\Uc_n\}\) is a right categorical cover, with
transformations
\(\alpha^i\colon\overline{c_i}\Rightarrow\iota_i\).

If \(s\) is a source, then \(s\in\Uc_i\) for some \(i\), since \(1_s\)
belongs to the cover. The component \(\alpha^i_s\colon c_i\to s\) forces
\(c_i=s\). Thus every source occurs among \(c_0,\ldots,c_n\), so
\(|\operatorname{Src}(\CC)|\leq n+1\). Taking a minimal cover gives the
result.
\end{proof}

For finite posets, this lower bound is attained.

\begin{prop}\label{prop:dcat_finite_poset}
Let \(P\) be a finite non-empty poset, regarded as a category, and let
\(r\) be the number of its minimal elements. Then
\[
\mathrm{dcat}(P)=r-1.
\]
\end{prop}

\begin{proof}
The minimal elements of \(P\) are its sources, so
Proposition~\ref{prop:dcat_sources_lower_bound} gives
\(\mathrm{dcat}(P)\geq r-1\).

Let \(s_1,\ldots,s_r\) be the minimal elements, and let \(\Uc_j\) be the
full subcategory on the principal upper set
\[
\uparrow s_j=\{x\in P\mid s_j\leq x\}.
\]
Every element of \(P\) lies above a minimal element. Moreover, if
\(x\leq y\) and \(s_j\leq x\), then \(s_j\leq y\). Hence the
\(\Uc_j\) cover all morphisms of \(P\).

For every \(x\in\Uc_j\), the unique morphism \(s_j\to x\) defines the
component of a natural transformation
\(\overline{s_j}\Rightarrow\iota_j\). Thus the \(\Uc_j\) form a right
categorical cover and \(\mathrm{dcat}(P)\leq r-1\).
\end{proof}

We next give an upper bound in terms of maximal chains. An \(n\)-chain in
\(\CC\) is a sequence
\[
c_0\xrightarrow{f_1}c_1\xrightarrow{f_2}\cdots
\xrightarrow{f_n}c_n.
\]
Equivalently, it is an \(n\)-simplex of the nerve \(N\CC\)
\cite{nlab:nerve}. It is \emph{non-degenerate} if none of the \(f_i\) is
an identity; objects are regarded as non-degenerate \(0\)-chains. A
non-degenerate chain is \emph{maximal} if it is not a proper face of
another non-degenerate chain. We denote the set of maximal chains by
\(\operatorname{MaxChains}(\CC)\).

\begin{prop}\label{prop:bounds_maxchains}
Let \(\CC\) be a finite non-empty acyclic category. Then
\[
\mathrm{dcat}(\CC)
\leq
|\operatorname{MaxChains}(\CC)|-1.
\]
\end{prop}

\begin{proof}
Every non-degenerate chain lies in a maximal one. For a maximal chain
\[
\bar f=
\bigl(
c_0\xrightarrow{f_1}c_1\xrightarrow{f_2}\cdots
\xrightarrow{f_n}c_n
\bigr),
\]
let \(\Uc_{\bar f}\) be the subcategory generated by its morphisms.

For every \(c_i\in\Uc_{\bar f}\), the composite
\(f_i\circ\cdots\circ f_1\colon c_0\to c_i\) defines the component of a
natural transformation
\(\overline{c_0}\Rightarrow\iota_{\bar f}\). Thus
\(\Uc_{\bar f}\) is right categorical.

Every non-identity morphism is a non-degenerate \(1\)-chain and hence lies
in a maximal chain. Every object, regarded as a \(0\)-chain, also lies in a
maximal chain. Therefore the subcategories \(\Uc_{\bar f}\) form a right
categorical cover, and
\[
\mathrm{dcat}(\CC)+1
\leq
|\operatorname{MaxChains}(\CC)|.
\]
\end{proof}

For an acyclic category \(\CC\), its subdivision
\(\operatorname{sub}(\CC)\) is the poset of non-degenerate chains ordered
by the face relation \cite{Subdivision}. In the undirected setting,
subdivision does not increase categorical LS-category:
\(\ccat(\operatorname{sub}(\CC))\leq\ccat(\CC)\)
\cite[Corollary~3.4 and Proposition~2.19]{LS-Tanaka}. The directed
invariant satisfies the following inequality in the opposite direction.

\begin{corollary}\label{cor:dcat_subdivision}
Let \(\CC\) be a finite non-empty acyclic category. Then
\[
\mathrm{dcat}(\CC)
\leq
\mathrm{dcat}\bigl(\operatorname{sub}(\CC)^{\op}\bigr).
\]
\end{corollary}

\begin{proof}
The sources of \(\operatorname{sub}(\CC)^{\op}\) are the maximal chains of
\(\CC\). By Proposition~\ref{prop:dcat_finite_poset},
\[
\mathrm{dcat}\bigl(\operatorname{sub}(\CC)^{\op}\bigr)
=
|\operatorname{MaxChains}(\CC)|-1.
\]
Now apply Proposition~\ref{prop:bounds_maxchains}.
\end{proof}

\subsection{Comparison with Undirected LS-Category}

To isolate the effect of orienting the homotopies while retaining the same
covers, we introduce the following undirected invariant.

\begin{definition}\label{def:strong_ccat}
The \emph{strong normalized categorical Lusternik--Schnirelmann category}
of a small category \(\CC\), denoted by \(\ccat(\CC)\), is the least
integer \(n\geq0\) for which \(\CC\) admits a cover
\(\{\Uc_0,\ldots,\Uc_n\}\) such that every inclusion
\(\iota_i\colon\Uc_i\hookrightarrow\CC\) is strongly homotopic to a
constant functor. If no finite cover exists, we set
\(\ccat(\CC)=\infty\).
\end{definition}

\begin{remark}
Although Definition~\ref{def:strong_ccat} uses the usual undirected
homotopy relation, its covers are the morphism covers of
Definition~\ref{def:categorical_cover}. Thus, \(\ccat\) differs from
Tanaka's categorical LS-category, which uses geometric covers
\cite{LS-Tanaka}.
\end{remark}

\begin{prop}\label{prop:ccat_leq_dcat}
For every small category \(\CC\),
\(\ccat(\CC)\leq\mathrm{dcat}(\CC)\).
\end{prop}

\begin{proof}
A natural transformation
\(\overline{c_i}\Rightarrow\iota_i\) is a zigzag of length one. Hence every
right categorical cover is also strongly categorical.
\end{proof}

\subsection{The Case of Monoids}
\label{subsec:dcat_monoids}

Let \(M\) be a monoid regarded as a one-object category. Its subcategories
are precisely its submonoids, and a cover of \(M\) is therefore a family of
submonoids whose union is \(M\).

\begin{prop}\label{prop:categorical_submonoids}
Let \(N\subseteq M\) be a submonoid. Then \(N\) is right categorical in
\(M\) if and only if there exists \(s\in M\) such that \(ns=s\) for every
\(n\in N\).
\end{prop}

\begin{proof}
There is a unique constant functor
\(\overline{\ast}\colon N\to M\). A natural transformation
\(\alpha\colon\overline{\ast}\Rightarrow\iota_N\) is determined by its
component \(s=\alpha_\ast\in M\). Naturality with respect to \(n\in N\)
is the commutativity of
\[
\begin{tikzcd}
\ast \arrow[r, "s"] \arrow[d, "1"'] &
\ast \arrow[d, "n"] \\
\ast \arrow[r, "s"'] &
\ast ,
\end{tikzcd}
\]
which is equivalent to \(ns=s\).
\end{proof}

\begin{remark}
The element \(s\) need not belong to \(N\). It is an element of the ambient
monoid fixed under left multiplication by \(N\). If \(s\in N\), then \(s\)
is a right absorbing element of \(N\).
\end{remark}

\begin{prop}\label{prop:dcat_monoid_cover}
Let \(M\) be a monoid. Then \(\mathrm{dcat}(M)\leq r-1\) if and only if
there exist submonoids \(N_1,\ldots,N_r\subseteq M\) and elements
\(s_1,\ldots,s_r\in M\) such that
\[
M=N_1\cup\cdots\cup N_r
\qquad\text{and}\qquad
ns_i=s_i
\quad
\text{for every }n\in N_i.
\]
\end{prop}

\begin{proof}
This follows directly from
Proposition~\ref{prop:categorical_submonoids}: the \(N_i\) form a right
categorical cover precisely under the stated conditions.
\end{proof}

\begin{corollary}\label{cor:dcat_zero_monoid}
A monoid \(M\) satisfies \(\mathrm{dcat}(M)=0\) if and only if it has a
right absorbing element.
\end{corollary}

\begin{proof}
Apply Proposition~\ref{prop:dcat_monoid_cover} with \(r=1\).
\end{proof}

\begin{prop}\label{prop:dcat_group}
Let \(G\) be a group regarded as a one-object category. Then
\[
\mathrm{dcat}(G)
=
\begin{cases}
0, & \text{if }G=\{1\},\\
\infty, & \text{if }G\neq\{1\}.
\end{cases}
\]
\end{prop}

\begin{proof}
The trivial group is the terminal category. If \(N\subseteq G\) is right
categorical, there exists \(s\in G\) such that \(ns=s\) for every
\(n\in N\). Cancellation gives \(n=1\), so \(N=\{1\}\). Hence no family of
right categorical submonoids can cover a nontrivial group.
\end{proof}

\subsection{Examples}

The following examples exhibit the asymmetry of directed LS-category.

\begin{example}\label{ex:Opposite}
Let \(\mathbb I_2\) be the category generated by
\[
\begin{tikzcd}
a_1 \arrow[rd, "f"] & \\
a_2 \arrow[r, "g"'] & b .
\end{tikzcd}
\]
The category \(\mathbb I_2\) has two sources and two maximal chains.
Propositions~\ref{prop:dcat_sources_lower_bound} and
\ref{prop:bounds_maxchains} therefore give
\(\mathrm{dcat}(\mathbb I_2)=1\).

By contrast, \(b\) is initial in \(\mathbb I_2^{\op}\), so
\(\mathrm{dcat}(\mathbb I_2^{\op})=0\). Thus directed LS-category is not
invariant under passage to the opposite category.
\end{example}

\begin{example}\label{ex:arbitrary_gap_LS}
For \(n\geq1\), let \(\mathbb I_n\) be the category generated by
\[
\begin{tikzcd}[column sep=1.5em, row sep=0.4em]
a_1 \arrow[rdd, "f_1"] & \\
\vdots                 & \\
a_i \arrow[r, "f_i"]   & b \\
\vdots                 & \\
a_n \arrow[ruu, "f_n"'] &
\end{tikzcd}
\]
The category has \(n\) sources and \(n\) maximal chains, so
\(\mathrm{dcat}(\mathbb I_n)=n-1\).

On the other hand, \(b\) is terminal. Hence
\(1_{\mathbb I_n}\Rightarrow\overline b\), so
\(\ccat(\mathbb I_n)=0\). Therefore
\[
\mathrm{dcat}(\mathbb I_n)-\ccat(\mathbb I_n)=n-1,
\]
and the difference between the directed and undirected invariants is
unbounded.
\end{example}

\begin{example}\label{ex:P_LS_Dcat}
Let \(\mathcal P\) be the category generated by
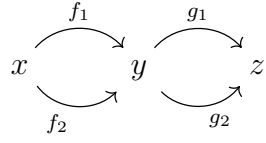
\[
\begin{tikzcd}
x
  \arrow[r, "f_1", bend left=49]
  \arrow[r, "f_2"', bend right=49]
&
y
  \arrow[r, "g_1", bend left=49]
  \arrow[r, "g_2"', bend right=49]
&
z
\end{tikzcd}
\]
subject to
\[
g_j\circ f_i=g_l\circ f_k
\qquad
(i,j,k,l\in\{1,2\}).
\]
Thus all four composites from \(x\) to \(z\) coincide.

The object \(x\) is the unique source but is not initial, since
\(f_1\neq f_2\). Hence
\(\mathrm{dcat}(\mathcal P)\neq0\).

For \(i\in\{1,2\}\), let \(\Uc_i\) be the subcategory generated by
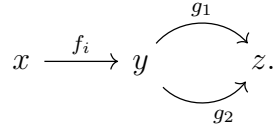
\[
\begin{tikzcd}
x \arrow[r, "f_i"]
&
y
  \arrow[r, "g_1", bend left=49]
  \arrow[r, "g_2"', bend right=49]
&
z .
\end{tikzcd}
\]
The subcategories \(\Uc_1\) and \(\Uc_2\) cover \(\mathcal P\). Moreover,
\(x\) is initial in each \(\Uc_i\): the unique morphism \(x\to y\) is
\(f_i\), and the defining relations imply
\(g_1f_i=g_2f_i\), giving a unique morphism \(x\to z\). Thus both
subcategories are right categorical, and
\[
\mathrm{dcat}(\mathcal P)=1.
\]
\end{example}

\section{Directed Sectional Category}
\label{sec:directed_sectional_category}

Sectional category measures the local existence of sections of a functor.
We first recall the strict version introduced in
\cite{Baues-Isaac}.

\begin{definition}\label{def:strict_sectional_category}
The \emph{sectional category} of a functor \(P\colon\E\to\Ba\), denoted by
\(\mathrm{secat}(P)\), is the least integer \(n\geq0\) for which \(\Ba\)
admits a cover \(\{\Uc_0,\ldots,\Uc_n\}\) and, for every \(i\), a functor
\(s_i\colon\Uc_i\to\E\) satisfying \(P\circ s_i=\iota_i\), where
\(\iota_i\colon\Uc_i\hookrightarrow\Ba\) is the inclusion.

The functor \(s_i\) is called a \emph{local section} of \(P\) over
\(\Uc_i\). If no such finite cover exists, we set
\(\mathrm{secat}(P)=\infty\).
\end{definition}

For an arbitrary functor, strict local sections may be too rigid.
Replacing them by homotopy sections gives the categorical analogue of the
\v{S}varc genus considered in
\cite{Baues-Isaac,carcaciacampos2026weakstrongfibrationsfunctors}.

We use right directed homotopies, since a database determines a discrete
opfibration and Grothendieck opfibrations satisfy the right directed
homotopy lifting property. The left-handed theory is dual.

\begin{definition}\label{def:directed_sectional_category}
The \emph{directed sectional category} of a functor
\(P\colon\E\to\Ba\), denoted by \(\mathrm{dsecat}(P)\), is the least
integer \(n\geq0\) for which \(\Ba\) admits a cover
\(\{\Uc_0,\ldots,\Uc_n\}\) and, for every \(i\), a functor
\(s_i\colon\Uc_i\to\E\) together with a natural transformation
\[
\alpha^i\colon P\circ s_i\Rightarrow\iota_i.
\]
If no such finite cover exists, we set \(\mathrm{dsecat}(P)=\infty\).
\end{definition}

The pair \((s_i,\alpha^i)\) is called a \emph{local right homotopy
section}. Equivalently, its defining condition is
\(P\circ s_i\leq_d\iota_i\).

\subsection{Composition}

Directed sectional category is monotone under composition: every local
right homotopy section of a composite \(PF\) induces one of \(P\).

\begin{prop}\label{prop:dsecat_composition}
Let \(P\colon\E\to\Ba\) and \(F\colon\CC\to\E\). Then
\[
\mathrm{dsecat}(P)
\leq
\mathrm{dsecat}(P F).
\]
\end{prop}

\begin{proof}
If \(s_i\colon\Uc_i\to\CC\) and
\(\alpha^i\colon PF s_i\Rightarrow\iota_i\) are local right homotopy
sections of \(PF\), then \(F s_i\colon\Uc_i\to\E\) and the same
transformation \(\alpha^i\) are local right homotopy sections of \(P\).
\end{proof}

\subsection{Opfibrations and Strictification}

For Grothendieck opfibrations, local right homotopy sections can be
strictified.

\begin{prop}\label{prop:secat_dsecat_comparison}
For every functor \(P\colon\E\to\Ba\),
\[
\mathrm{dsecat}(P)\leq\mathrm{secat}(P).
\]
If \(P\) is a Grothendieck opfibration, then
\[
\mathrm{dsecat}(P)=\mathrm{secat}(P).
\]
\end{prop}

\begin{proof}
Every strict local section is a local right homotopy section, using the
identity transformation. Hence
\(\mathrm{dsecat}(P)\leq\mathrm{secat}(P)\).

Suppose that \(P\) is a Grothendieck opfibration and that
\(s\colon\Uc\to\E\) is a local right homotopy section, with
\(\alpha\colon P s\Rightarrow\iota\). Let
\(H\colon\Uc\times\mathbb I_1\to\Ba\) be the homotopy corresponding to
\(\alpha\). We have a commutative square
\[
\begin{tikzcd}
\Uc \arrow[r, "s"] \arrow[d, "\iota_0"'] &
\E \arrow[d, "P"] \\
\Uc\times\mathbb I_1 \arrow[r, "H"'] &
\Ba .
\end{tikzcd}
\]

By Proposition~\ref{prop:Grothendieck_implies_directed}, \(P\) is a right
directed fibration. Hence \(H\) admits a lift
\(\widetilde H\colon\Uc\times\mathbb I_1\to\E\) satisfying
\(\widetilde H\iota_0=s\). The functor
\(\overline{s}=\widetilde H\iota_1\) satisfies
\[
P\overline{s}
=
P\widetilde H\iota_1
=
H\iota_1
=
\iota,
\]
and is therefore a strict local section. Thus every local right homotopy
section of \(P\) can be strictified.
\end{proof}

\begin{corollary}\label{cor:dsecat_zero_opfibration}
A Grothendieck opfibration \(P\colon\E\to\Ba\) satisfies
\(\mathrm{dsecat}(P)\allowbreak=\allowbreak0\) if and only if it admits a global section.
\end{corollary}

\begin{proof}
This follows immediately from
Proposition~\ref{prop:secat_dsecat_comparison}.
\end{proof}

\subsection{Monotonicity and Homotopy Invariance}

Directed sectional category is monotone with respect to the directed
homotopy preorder. Indeed, a directed deformation from \(P\) to \(P'\)
allows every local right homotopy section of \(P'\) to be regarded as one
of \(P\). We first record this elementary observation and then use it to
obtain an invariance result for functors whose total and base categories
are related by compatible directed comparison data.

\begin{prop}\label{prop:dsecat_monotone_homotopy}
Let \(P,P'\colon\E\to\Ba\) be functors. If \(P\leq_dP'\), then
\[
\mathrm{dsecat}(P)\leq\mathrm{dsecat}(P').
\]
\end{prop}

\begin{proof}
Let \(\eta\colon P\Rightarrow P'\), and suppose that
\(s_i\colon\Uc_i\to\E\) is a local right homotopy section of \(P'\).
Then
\[
Ps_i
\leq_d
P's_i
\leq_d
\iota_i.
\]
By transitivity, \(s_i\) is a local right homotopy section of \(P\).
Thus every cover witnessing \(\mathrm{dsecat}(P')\leq n\) also witnesses
\(\mathrm{dsecat}(P)\leq n\).
\end{proof}

Proposition~\ref{prop:dsecat_monotone_homotopy} compares functors with the same
domain and codomain. We now allow both categories to vary. To transfer
local sections, we require functors between the total categories and
between the bases, together with directed homotopies expressing that the
resulting squares commute up to the prescribed orientation. Mutual right
dominations of the base categories then provide the comparison in both
directions.

\begin{theorem}\label{thr:dsecat_invariance_lax_equivalence}
Consider a diagram
\[
\begin{tikzcd}[column sep=2em, row sep=2.5em]
\E
  \arrow[r, "F_1", bend left=35]
  \arrow[d, "P"']
&
\E'
  \arrow[l, "G_1", bend left=35]
  \arrow[d, "P'"]
\\
\Ba
  \arrow[r, "F_2", bend left=35]
&
\Ba'
  \arrow[l, "G_2", bend left=35]
\end{tikzcd}
\]
together with natural transformations
\[
\lambda\colon P'F_1\Rightarrow F_2P,
\qquad
\mu\colon PG_1\Rightarrow G_2P'.
\]
Suppose that there are also natural transformations
\[
\alpha\colon G_2F_2\Rightarrow1_\Ba,
\qquad
\beta\colon F_2G_2\Rightarrow1_{\Ba'}.
\]
Then
\[
\mathrm{dsecat}(P)=\mathrm{dsecat}(P').
\]
\end{theorem}

\begin{proof}
We first show that
\(\mathrm{dsecat}(P)\leq\mathrm{dsecat}(P')\). Let
\(\{\Uc_0,\ldots,\Uc_n\}\) be a cover of \(\Ba'\) admitting local right
homotopy sections
\[
s_i'\colon\Uc_i\longrightarrow\E',
\qquad
P's_i'\leq_d\iota_i'.
\]
Set \(\Vc_i=F_2^{-1}(\Uc_i)\), let
\(F_{2,i}\colon\Vc_i\to\Uc_i\) be the restriction of \(F_2\), and define
\[
s_i=G_1s_i'F_{2,i}\colon\Vc_i\longrightarrow\E.
\]
The given transformations and the transitivity of \(\leq_d\) yield
\[
Ps_i=PG_1s_i'F_{2,i}
\leq_d
G_2P's_i'F_{2,i}
\leq_d
G_2\iota_i'F_{2,i}
=
G_2F_2\iota_i
\leq_d
\iota_i.
\]
Thus \(s_i\) is a local right homotopy section of \(P\). Since the
\(\Vc_i\) cover \(\Ba\), we obtain
\[
\mathrm{dsecat}(P)\leq\mathrm{dsecat}(P').
\]
For the reverse inequality, start with a cover of \(\Ba\), take its
inverse images under \(G_2\), and use \(F_1\), \(\lambda\), and \(\beta\).
\end{proof}

\begin{remark}
The strictly commutative case is recovered by taking \(\lambda\) and
\(\mu\) to be identity transformations.
\end{remark}

The principal application is that the opfibrational replacement of an
arbitrary functor preserves directed sectional category. 

\begin{corollary}\label{cor:dsecat_opfibration_factorization}
Let \(P\colon\E\to\Ba\), and consider the factorization
\[
\begin{tikzcd}
\E \arrow[rr, "P"] \arrow[dr, "P_1"'] &&
\Ba \\
& E_P \arrow[ur, "P_2"'] &
\end{tikzcd}
\]
of Proposition~\ref{prop:directed_opfibration_factorization}. Then
\[
\mathrm{dsecat}(P)
=
\mathrm{dsecat}(P_2)
=
\mathrm{secat}(P_2).
\]
\end{corollary}

\begin{proof}
Let \(R\colon E_P\to\E\) and
\(\eta\colon P_1R\Rightarrow1_{E_P}\) be those constructed in
Proposition~\ref{prop:directed_opfibration_factorization}. Apply
Theorem~\ref{thr:dsecat_invariance_lax_equivalence} with
\[
F_1=P_1,\qquad
G_1=R,\qquad
F_2=G_2=1_\Ba.
\]
The required compatibility transformations are the identity
\(P_2P_1=P\) and the transformation
\[
PR=P_2P_1R\Rightarrow P_2
\]
obtained by applying \(P_2\) to \(\eta\). Hence
\(\mathrm{dsecat}(P)=\mathrm{dsecat}(P_2)\). Since \(P_2\) is a
Grothendieck opfibration, Proposition~\ref{prop:secat_dsecat_comparison}
gives the second equality.
\end{proof}

\subsection{Component sectional category}\label{subsec:component_sectional}

The opfibrational replacement retains the full directed sectional
obstruction. We now compare it with the discrete model provided by the
comprehensive factorization.

Let \(P\colon\E\to\Ba\) be a functor. Recall from
Subsection~\ref{subsec:discrete_factorization} that its comprehensive
factorization is
\[
\begin{tikzcd}
\E
  \arrow[rr, "P"]
  \arrow[dr, "i_P"']
&&
\Ba
\\
&
\displaystyle\int^\Ba K_P
  \arrow[ur, "q_P"']
&
\end{tikzcd}
\]
where
\[
K_P(b)=\pi_0(P\downarrow b),
\]
the functor \(i_P\) is initial, and \(q_P\) is a discrete opfibration.

The database \(K_P\) retains only the connected components of the comma
categories \(P\downarrow b\). It therefore provides a discrete
approximation to the local section problem for \(P\).

\begin{definition}\label{def:component_sectional_category}
The \emph{component sectional category} of a functor
\(P\colon\E\to\Ba\) is
\[
\mathrm{csecat}(P)
:=
\mathrm{dsecat}(q_P).
\]
\end{definition}

Since \(q_P\) is a discrete opfibration, its directed sectional category
agrees with its strict sectional category. Thus
\(\mathrm{csecat}(P)\) measures the minimum normalized number of
subcategories needed to cover \(\Ba\) so that compatible connected
components of the comma categories \(P\downarrow b\) can be selected
locally.

\begin{proposition}\label{prop:csecat_lower_bound}
For every functor \(P\colon\E\to\Ba\),
\[
\mathrm{csecat}(P)
\leq
\mathrm{dsecat}(P).
\]
\end{proposition}

\begin{proof}
The comprehensive factorization satisfies \(P=q_Pi_P\). Hence
Proposition~\ref{prop:dsecat_composition} gives
\[
\mathrm{dsecat}(q_P)
\leq
\mathrm{dsecat}(q_Pi_P)
=
\mathrm{dsecat}(P).
\]
\end{proof}

The preceding inequality admits a description purely in terms of limits.

\begin{proposition}\label{prop:csecat_local_limits}
Let \(P\colon\E\to\Ba\) be a functor. Then
\(\mathrm{csecat}(P)\leq n\) if and only if \(\Ba\) admits a cover
\(\{\Uc_0,\ldots,\Uc_n\}\) such that
\[
\lim_{\Uc_i}K_P|_{\Uc_i}
\neq
\varnothing
\]
for every \(i\).
\end{proposition}

\begin{proof}
Since \(q_P\) is the category-of-elements projection of \(K_P\), a section
of \(q_P\) over \(\Uc_i\) is equivalent to an element of
\[
\lim_{\Uc_i}K_P|_{\Uc_i}.
\]
Moreover, \(q_P\) is a discrete opfibration, so local right homotopy
sections can be strictified by
Proposition~\ref{prop:secat_dsecat_comparison}. The result follows from the
definition of \(\mathrm{csecat}(P)\).
\end{proof}

Suppose now that \(P\) is the Grothendieck opfibration associated with a
functor \(X\colon\Ba\to\cat\). As shown in
Proposition~\ref{prop:comprehensive_grothendieck}, its component database is
naturally isomorphic to
\[
\pi_0X\colon\Ba\longrightarrow\Sets,
\qquad
b\longmapsto\pi_0(X(b)).
\]

\begin{corollary}\label{cor:csecat_grothendieck}
Let \(X\colon\Ba\to\cat\), and let
\[
P_X\colon\int^\Ba X\longrightarrow\Ba
\]
be its Grothendieck opfibration. Then
\[
\mathrm{csecat}(P_X)
=
\mathrm{dsecat}(\pi_{\pi_0X})
\leq
\mathrm{dsecat}(P_X).
\]
\end{corollary}

\begin{proof}
The discrete opfibration in the comprehensive factorization of \(P_X\) is
naturally isomorphic to
\[
\pi_{\pi_0X}\colon
\int^\Ba\pi_0X
\longrightarrow
\Ba.
\]
The equality follows from
Definition~\ref{def:component_sectional_category}, and the inequality from
Proposition~\ref{prop:csecat_lower_bound}.
\end{proof}

Combining the opfibrational and comprehensive factorizations gives the
following comparison.

\begin{theorem}\label{thm:component_and_opfibrational_models}
Let \(P\colon\E\to\Ba\) be a functor. Let
\[
\E\xrightarrow{P_1}\E_P\xrightarrow{P_2}\Ba
\]
be its directed opfibrational factorization, and let
\[
\E\xrightarrow{i_P}\int^\Ba K_P\xrightarrow{q_P}\Ba
\]
be its comprehensive factorization. Then
\[
\mathrm{csecat}(P)
=
\mathrm{dsecat}(q_P)
\leq
\mathrm{dsecat}(P)
=
\mathrm{dsecat}(P_2).
\]
\end{theorem}

\begin{proof}
The inequality is
Proposition~\ref{prop:csecat_lower_bound}. The final equality follows from
Corollary~\ref{cor:dsecat_opfibration_factorization}.
\end{proof}

Thus every functor admits two canonical opfibrational models. The
Grothendieck opfibration \(P_2\) retains the full directed sectional
obstruction, whereas the discrete opfibration \(q_P\) records only the
obstruction visible at the level of connected components. 

\subsection{Relations with Directed LS-Category}

\begin{prop}\label{inequality_genus_LS}
Let \(P\colon\E\to\Ba\) be surjective on objects. Then
\[
\mathrm{dsecat}(P)\leq\mathrm{dcat}(\Ba).
\]
\end{prop}

\begin{proof}
Let \(\{\Uc_0,\ldots,\Uc_n\}\) be a right categorical cover of \(\Ba\),
with transformations
\(\alpha^i\colon\overline{b_i}\Rightarrow\iota_i\).
Choose \(e_i\in\E\) with \(P(e_i)=b_i\), and let
\(s_i\colon\Uc_i\to\E\) be constant at \(e_i\). Then \(Ps_i\) is constant
at \(b_i\), so \(\alpha^i\colon Ps_i\Rightarrow\iota_i\). Thus the same
cover witnesses \(\mathrm{dsecat}(P)\leq\mathrm{dcat}(\Ba)\).
\end{proof}

For finite posets, it is enough that the fibres over the sources be
non-empty.

\begin{corollary}\label{cor:dsecat_poset_sources}
Let \(P\colon\E\to Q\), where \(Q\) is a finite non-empty poset. If
\(P^{-1}(s)\neq\varnothing\) for every source \(s\) of \(Q\), then
\[
\mathrm{dsecat}(P)\leq\mathrm{dcat}(Q).
\]
\end{corollary}

\begin{proof}
Let \(s_1,\ldots,s_r\) be the sources of \(Q\). The principal upper sets
\(\Uc_i=\uparrow s_i\) form a right categorical cover. Choose
\(e_i\in P^{-1}(s_i)\). The constant functor at \(e_i\) is a local right
homotopy section over \(\Uc_i\), since \(s_i\) is initial in \(\Uc_i\).
Hence \(\mathrm{dsecat}(P)\leq r-1=\mathrm{dcat}(Q)\).
\end{proof}

\subsection{Relations with the \v{S}varc Genus}
\label{subsec:svarc_genus}

To isolate the effect of orienting the local homotopies while retaining the
same notion of cover, we introduce the corresponding undirected invariant.

\begin{definition}\label{def:strong_sectional_category}
Let \(P\colon\E\to\Ba\) be a functor. The \emph{strong normalized
\v{S}varc genus} of \(P\), denoted by \(\mathrm{Sg}(P)\), is the
least integer \(n\geq0\) for which \(\Ba\) admits a cover
\(\{\Uc_0,\ldots,\Uc_n\}\) such that, for every \(i\), there exists a
functor \(s_i\colon\Uc_i\to\E\) for which \(P s_i\) is strongly homotopic
to the inclusion \(\iota_i\colon\Uc_i\hookrightarrow\Ba\). If no such
finite cover exists, we set \(\mathrm{Sg}(P)=\infty\).
\end{definition}

Thus, the local condition is the existence of a zigzag of natural
transformations connecting \(P s_i\) and \(\iota_i\), with no prescribed
orientation.

\begin{remark}
Definition~\ref{def:strong_sectional_category} uses the morphism covers of
Definition~\ref{def:categorical_cover}. It should therefore be
distinguished from categorical versions of the \v{S}varc genus defined
using geometric covers
\cite{Baues-Isaac,carcaciacampos2026weakstrongfibrationsfunctors}.
\end{remark}

\begin{proposition}\label{prop:ssecat_leq_dsecat}
For every functor \(P\colon\E\to\Ba\),
\[
\mathrm{Sg}(P)
\leq
\mathrm{dsecat}(P).
\]
\end{proposition}

\begin{proof}
A local right homotopy section consists of a functor
\(s_i\colon\Uc_i\to\E\) and a natural transformation
\(P s_i\Rightarrow\iota_i\). This is, in particular, a zigzag of natural
transformations between \(P s_i\) and \(\iota_i\). Hence every cover
admitting local right homotopy sections also admits local strong homotopy
sections.
\end{proof}

Combining Proposition~\ref{prop:ssecat_leq_dsecat} with
Proposition~\ref{prop:secat_dsecat_comparison}, we obtain
\[
\mathrm{Sg}(P)
\leq
\mathrm{dsecat}(P)
\leq
\mathrm{secat}(P).
\]

A \emph{Grothendieck bifibration} is a functor that is both a
Grothendieck fibration and a Grothendieck opfibration. In this case, the
three invariants coincide.

\begin{proposition}\label{prop:bifibration_sectional_categories}
Let \(P\colon\E\to\Ba\) be a Grothendieck bifibration. Then
\[
\mathrm{Sg}(P)
=
\mathrm{dsecat}(P)
=
\mathrm{secat}(P).
\]
\end{proposition}

\begin{proof}
The strictification argument of
\cite[Proposition~2.18]{Baues-Isaac} shows that, for a bifibration, every
local strong homotopy section can be replaced by a strict local section. The argument successively lifts the arrows in the zigzag, using the
fibrational or opfibrational structure according to their orientation,
until the homotopy section is replaced by a strict one.
Although that result is stated using geometric covers, its proof applies
independently to each member of the cover and therefore remains valid for
the morphism covers used here.  Hence
\[
\mathrm{Sg}(P)=\mathrm{secat}(P).
\]

Moreover, every Grothendieck bifibration is, in particular, a Grothendieck
opfibration. Proposition~\ref{prop:secat_dsecat_comparison} therefore gives
\[
\mathrm{dsecat}(P)=\mathrm{secat}(P),
\]
and the result follows.
\end{proof}

\subsection{Examples}\label{sec:examples_sectional}

\begin{example}\label{example:monoid_action_sectional_category}
Let \(X\colon M\to\Sets\) be a monoid action, as in
Example~\ref{example:monoid_action_database}. A cover of the one-object
category \(M\) is precisely a family of submonoids whose union is \(M\).
Since \(\pi_X\) is a discrete opfibration, its local right homotopy
sections can be strictified.

Consequently, \(\mathrm{dsecat}(\pi_X)\leq r-1\) if and only if there
exist submonoids \(M_1,\ldots,M_r\subseteq M\) and, for each \(i\), an
element \(x_i\in X(\ast)\) such that
\[
M=M_1\cup\cdots\cup M_r
\qquad\text{and}\qquad
m\cdot x_i=x_i
\quad
\text{for every }m\in M_i.
\]
The elements \(x_i\) may be different for different submonoids. Thus,
\(\mathrm{dsecat}(\pi_X)+1\) is the minimum number of submonoids covering
\(M\) such that the action of each \(M_i\) has an \(M_i\)-fixed point. In
particular, \(\mathrm{dsecat}(\pi_X)=0\) if and only if the action of
\(M\) has a global fixed point.
\end{example}

The following example shows the difference between the directed sectional category and the strong \v{S}varc genus.

\begin{example}\label{ex:Sg_dsecat_gap}
Let \(\mathcal P\) be the category of
Example~\ref{ex:P_LS_Dcat}, generated by
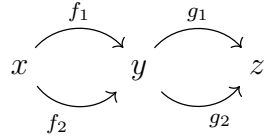
\[
\begin{tikzcd}
x
  \arrow[r, "f_1", bend left=49]
  \arrow[r, "f_2"', bend right=49]
&
y
  \arrow[r, "g_1", bend left=49]
  \arrow[r, "g_2"', bend right=49]
&
z
\end{tikzcd}
\]
subject to the relations
\[
g_jf_i=g_lf_k
\qquad
(i,j,k,l\in\{1,2\}).
\]
Recall that \(\mathrm{dcat}(\mathcal P)=1\).

Let \(\Dc\) be the free category generated by
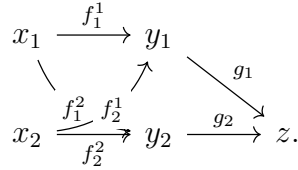
\[
\begin{tikzcd}
x_1
  \arrow[r, "f_1^1"]
  \arrow[rd, "f_1^2" description, bend right=30]
&
y_1 \arrow[rd, "g_1"]
&
\\
x_2
  \arrow[r, "f_2^2"']
  \arrow[ru, "f_2^1" description, bend right=30]
&
y_2 \arrow[r, "g_2"]
&
z .
\end{tikzcd}
\]
Define \(F\colon\Dc\to\mathcal P\) by
\[
F(x_i)=x,\qquad F(y_i)=y,\qquad F(z)=z,
\]
and
\[
F(f_i^j)=f_j,\qquad F(g_i)=g_i.
\]

Since \(F\) is surjective on objects,
Proposition~\ref{inequality_genus_LS} gives
\[
\mathrm{dsecat}(F)
\leq
\mathrm{dcat}(\mathcal P)
=
1.
\]

We claim that \(F\) has no global right homotopy section. Suppose otherwise
that there exist a functor \(s\colon\mathcal P\to\Dc\) and a natural
transformation
\[
\alpha\colon Fs\Rightarrow1_{\mathcal P}.
\]
Since \(x\) is a source, the component
\(\alpha_x\colon F(s(x))\to x\) forces
\(F(s(x))=x\). Hence \(s(x)=x_i\) for some \(i\in\{1,2\}\), and
\(\alpha_x=1_x\).

Naturality with respect to \(f_k\colon x\to y\), for \(k=1,2\), gives
\[
\alpha_y\circ F(s(f_k))=f_k.
\]
If \(F(s(y))=x\), then the existence of both morphisms
\(s(f_k)\colon s(x)\to s(y)\) forces \(s(y)=s(x)\) and
\(s(f_k)=1_{s(x)}\). The preceding equations would then give
\(\alpha_y=f_1=f_2\), a contradiction.

If \(F(s(y))=y\), then \(s(y)=y_j\) for some \(j\in\{1,2\}\), and
\(\alpha_y=1_y\). However, the unique morphism \(x_i\to y_j\) is
\(f_i^j\), whose image under \(F\) is \(f_j\). Thus both naturality
conditions would require
\[
f_j=f_1
\qquad\text{and}\qquad
f_j=f_2,
\]
again a contradiction. Finally, \(F(s(y))=z\) is impossible because there
is no morphism \(z\to y\) in \(\mathcal P\).

Therefore \(F\) has no global right homotopy section, so
\(\mathrm{dsecat}(F)\neq0\). Combining both inequalities yields
\[
\mathrm{dsecat}(F)=1.
\]

On the other hand, let
\(\overline z\colon\mathcal P\to\Dc\) be the constant functor at \(z\).
Then \(F\overline z\) is the constant functor at the terminal object \(z\)
of \(\mathcal P\). Hence there is a natural transformation
\[
1_{\mathcal P}\Rightarrow F\overline z.
\]
Thus \(F\overline z\) is strongly homotopic to \(1_{\mathcal P}\), and
\[
\mathrm{Sg}(F)=0.
\]
Consequently,
\[
\mathrm{Sg}(F)=0
<
\mathrm{dsecat}(F)=1.
\]
\end{example}

\section{Databases and directed sectional category}\label{sec:databases_sectional}

A database \(X\colon\CC\to\Sets\) determines a discrete opfibration
\[
\pi_X\colon\int^\CC X\longrightarrow\CC,
\]
whose sections correspond to globally coherent choices of records. Since
every discrete opfibration is a Grothendieck opfibration by
Proposition~\ref{prop:unique_lifts_are_opcartesian},
Proposition~\ref{prop:secat_dsecat_comparison} gives
\[
\mathrm{dsecat}(\pi_X)=\mathrm{secat}(\pi_X).
\]
Thus every local right homotopy section of \(\pi_X\) can be strictified.
The directed formulation nevertheless remains useful, since it places the
database problem within the homotopy-invariant framework developed in the
preceding section, while the strict formulation gives its concrete
interpretation in terms of coherent local selections.

In this section, we study directed sectional category for the discrete
opfibrations associated with functorial databases. We consider
indecomposable and iterated databases, analyse its behaviour under data
migration, and conclude by characterizing initial objects through the
universal existence of global sections.

\subsection{Connected Components and Indecomposable Databases}
\label{subsec:indecomposable_databases}

Coproducts of databases are computed objectwise. A non-empty database
\(X\colon\CC\to\Sets\) is called \emph{indecomposable} if every
decomposition \(X\cong X_1\coprod X_2\) has an empty summand. For this
notion and its connections with categorical actions, Burnside rings, and
biset functors, see \cite{webb2023bisetfunctorscategories}.

The indecomposable summands of a database are determined by the connected
components of its category of elements.

\begin{prop}\label{prop:database_connected_decomposition}
Let
\[
\int^\CC X
=
\coprod_{\lambda\in\Lambda}\E_\lambda
\]
be the decomposition of the category of elements of \(X\) into connected
components. Then there are indecomposable subfunctors
\(X_\lambda\subseteq X\) such that
\[
X\cong\coprod_{\lambda\in\Lambda}X_\lambda,
\qquad
\int^\CC X_\lambda\cong\E_\lambda.
\]
\end{prop}

\begin{proof}
For \(c\in\CC\), define
\[
X_\lambda(c)
=
\{x\in X(c)\mid(c,x)\in\E_\lambda\}.
\]
If \(f\colon c\to d\) and \(x\in X_\lambda(c)\), the morphism
\[
(c,x)\longrightarrow\bigl(d,X(f)(x)\bigr)
\]
shows that \(X(f)(x)\in X_\lambda(d)\). Hence \(X_\lambda\) is a
subfunctor.

Every object \((c,x)\) belongs to a unique connected component, so
\(X(c)=\coprod_\lambda X_\lambda(c)\), naturally in \(c\). Therefore
\(X\cong\coprod_\lambda X_\lambda\), and by construction
\(\int^\CC X_\lambda\cong\E_\lambda\). Since \(\E_\lambda\) is connected,
\(X_\lambda\) is indecomposable.
\end{proof}

\begin{corollary}\label{cor:indecomposable_connected_elements}
A non-empty database \(X\colon\CC\to\Sets\) is indecomposable if and only if
\(\int^\CC X\) is connected.
\end{corollary}

For connected schemas, the existence of a global section can therefore be
checked on the indecomposable summands.

\begin{prop}\label{prop:section_indecomposable_database}
Let \(\CC\) be connected, and let
\(X\cong\coprod_{\lambda\in\Lambda}X_\lambda\) be its decomposition into
indecomposable summands. Then \(\pi_X\) admits a global section if and only
if \(\pi_{X_\lambda}\) admits a global section for some \(\lambda\).
\end{prop}

\begin{proof}
The image of a section
\(s\colon\CC\to\int^\CC X\) is contained in a single connected component,
and hence \(s\) factors through some \(\int^\CC X_\lambda\). The converse
follows by composing a section of \(\pi_{X_\lambda}\) with the inclusion
\(\int^\CC X_\lambda\hookrightarrow\int^\CC X\).
\end{proof}

\begin{theorem}\label{prop:dsecat_indecomposable_summands}
Let
\[
X\cong\coprod_{\lambda\in\Lambda}X_\lambda
\]
be the decomposition of a database into indecomposable summands. Then, for
every \(\lambda\in\Lambda\),
\[
\mathrm{dsecat}(\pi_X)
\leq
\mathrm{dsecat}(\pi_{X_\lambda}).
\]
Consequently, if \(\Lambda\) is finite and non-empty, then
\[
\mathrm{dsecat}(\pi_X)
\leq
\min_{\lambda\in\Lambda}
\mathrm{dsecat}(\pi_{X_\lambda}).
\]
\end{theorem}

\begin{proof}
The inclusion \(X_\lambda\hookrightarrow X\) induces a functor
\[
j_\lambda\colon
\int^\CC X_\lambda
\longrightarrow
\int^\CC X
\]
over \(\CC\). Since
\(\pi_Xj_\lambda=\pi_{X_\lambda}\),
Proposition~\ref{prop:dsecat_composition} gives
\[
\mathrm{dsecat}(\pi_X)
\leq
\mathrm{dsecat}(\pi_{X_\lambda}).
\]
\end{proof}

\begin{remark}
A local section of \(\pi_X\) over a connected subcategory factors through
a unique indecomposable summand. Different members of a cover may,
however, use different summands. Therefore the preceding inequality need
not be an equality, and \(\mathrm{dsecat}(\pi_X)\) cannot in general be
recovered from the directed sectional category of a single summand.
\end{remark}

\subsection{Iterated Databases}\label{subsec:iterated}

We now apply Proposition~\ref{prop:dsecat_composition} to iterated
functorial databases. A database on the category of elements of another
database can be flattened into a single database on the original schema,
and the corresponding category-of-elements projections agree under this
identification.

Let \(X\colon\CC\to\Sets\), and let
\(Y\colon\int^\CC X\to\Sets\) be a database on its category of elements.
The standard equivalence
\[
\left[\int^\CC X,\Sets\right]
\simeq
[\CC,\Sets]/X
\]
identifies \(Y\) with a database over \(\CC\) equipped with a natural
transformation to \(X\).

Define the \emph{flattened database}
\(\Sigma_XY\colon\CC\to\Sets\) by
\[
(\Sigma_XY)(c)
=
\coprod_{x\in X(c)}Y(c,x).
\]
For \(f\colon c\to d\), set
\[
(\Sigma_XY)(f)(x,y)
=
\bigl(X(f)(x),Y(\overline f_x)(y)\bigr),
\]
where
\(\overline f_x\colon(c,x)\to(d,X(f)(x))\) is the unique morphism over
\(f\) in \(\int^\CC X\).

\begin{prop}\label{prop:iterated_Grothendieck_construction}
There is a canonical isomorphism
\[
\int^{\int^\CC X}Y
\cong
\int^\CC(\Sigma_XY)
\]
under which \(\pi_X\circ\pi_Y\) corresponds to
\(\pi_{\Sigma_XY}\).
\end{prop}

\begin{proof}
Define
\[
\Theta\bigl((c,x),y\bigr)
=
\bigl(c,(x,y)\bigr).
\]
The morphism conditions in the two categories agree by the definition of
\(\Sigma_XY\), so \(\Theta\) is an isomorphism. Both projections send the
corresponding object to \(c\) and the corresponding morphism to its
underlying morphism in \(\CC\).
\end{proof}

\begin{corollary}\label{cor:dsecat_iterated_database}
For databases \(X\colon\CC\to\Sets\) and
\(Y\colon\int^\CC X\to\Sets\),
\[
\mathrm{dsecat}(\pi_X)
\leq
\mathrm{dsecat}(\pi_{\Sigma_XY}).
\]
Equivalently,
\[
\mathrm{secat}(\pi_X)
\leq
\mathrm{secat}(\pi_{\Sigma_XY}).
\]
\end{corollary}

\begin{proof}
Proposition~\ref{prop:dsecat_composition} gives
\(\mathrm{dsecat}(\pi_X)\leq
\mathrm{dsecat}(\pi_X\pi_Y)\).
By Proposition~\ref{prop:iterated_Grothendieck_construction},
\(\pi_X\pi_Y\) is identified with \(\pi_{\Sigma_XY}\).
The equivalent statement for \(\mathrm{secat}\) follows because all the
projections involved are discrete opfibrations.
\end{proof}

\subsection{Behaviour under Data Migration}
\label{subsec:data_migration}

Let \(F\colon\CC\to\Dc\) be a functor between database schemas. It induces
three data migration functors
\[
\Sigma_F
\dashv
\Delta_F
\dashv
\Pi_F
\]
\cite{SpivakDataMigration}. Their types are
\[
\Delta_F\colon[\Dc,\Sets]\longrightarrow[\CC,\Sets],
\qquad
\Sigma_F,\Pi_F\colon[\CC,\Sets]\longrightarrow[\Dc,\Sets].
\]

The functor \(\Delta_F\) is given by precomposition:
\[
\Delta_FY=Y\circ F
\]
for every database \(Y\colon\Dc\to\Sets\). Thus, \(\Delta_F\) reindexes a
database on \(\Dc\) along \(F\), retaining only the entity types and
functional relationships visible through the schema \(\CC\).

The left and right migration functors are the left and right Kan
extensions along \(F\):
\[
\Sigma_F=\operatorname{Lan}_F,
\qquad
\Pi_F=\operatorname{Ran}_F.
\]
Pointwise, for a database \(X\colon\CC\to\Sets\) and an object
\(d\in\Dc\), they are given by
\[
(\Sigma_FX)(d)
\cong
\mathop{\operatorname{colim}}_{(F\downarrow d)}X\circ p_d
\]
and
\[
(\Pi_FX)(d)
\cong
\mathop{\operatorname{lim}}_{(d\downarrow F)}X\circ q_d,
\]
where
\[
p_d\colon(F\downarrow d)\longrightarrow\CC,
\qquad
q_d\colon(d\downarrow F)\longrightarrow\CC
\]
are the canonical projection functors.

Hence \(\Sigma_F\) transports data forward by collecting them through
colimits, whereas \(\Pi_F\) transports data forward by imposing the
compatibility conditions encoded by the corresponding limits. 

We now examine how these migrations interact with directed sectional
category. We first study restriction along \(F\), using the behaviour of directed sectional category under pullback.

\begin{prop}\label{prop:dsecat_base_change}
Consider a pullback square
\[
\begin{tikzcd}
F^\ast\E
  \arrow[r, "\widetilde F"]
  \arrow[d, "F^\ast P"']
&
\E
  \arrow[d, "P"]
\\
\CC
  \arrow[r, "F"']
&
\Dc .
\end{tikzcd}
\]
If \(P\) is a Grothendieck opfibration, then
\[
\mathrm{dsecat}(F^\ast P)
\leq
\mathrm{dsecat}(P).
\]
\end{prop}

\begin{proof}
Since \(P\) and \(F^\ast P\) are Grothendieck opfibrations, their directed
sectional categories agree with their strict sectional categories.

Let \(\{\Uc_0,\ldots,\Uc_n\}\) be a cover of \(\Dc\) admitting local
sections \(s_i\colon\Uc_i\to\E\) of \(P\). The inverse images
\(F^{-1}(\Uc_i)\) form a cover of \(\CC\), and every \(s_i\) pulls back to
a section
\[
F^\ast s_i\colon F^{-1}(\Uc_i)\longrightarrow F^\ast\E
\]
of \(F^\ast P\). Hence
\[
\mathrm{dsecat}(F^\ast P)
=
\mathrm{secat}(F^\ast P)
\leq
\mathrm{secat}(P)
=
\mathrm{dsecat}(P).
\]
\end{proof}

As a corollary, we obtain the following.\begin{corollary}\label{cor:dsecat_delta_migration}
Let \(F\colon\CC\to\Dc\) be a functor and let
\(X\colon\Dc\to\Sets\) be a database. Then
\[
\mathrm{dsecat}(\pi_{\Delta_FX})
\leq
\mathrm{dsecat}(\pi_X).
\]
\end{corollary}

\begin{proof}
There is a canonical pullback square
\[
\begin{tikzcd}
\int^\CC\Delta_FX
  \arrow[r, "\widetilde F"]
  \arrow[d, "\pi_{\Delta_FX}"']
&
\int^\Dc X
  \arrow[d, "\pi_X"]
\\
\CC
  \arrow[r, "F"']
&
\Dc ,
\end{tikzcd}
\]
where \(\widetilde F(c,x)=(F(c),x)\). Thus
\(\pi_{\Delta_FX}\cong F^\ast\pi_X\), and the result follows from
Proposition~\ref{prop:dsecat_base_change}.
\end{proof}

Thus, restriction along a schema functor cannot increase the sectional
obstruction. This is consistent with the fact that \(\Delta_F\) may forget
some of the compatibility conditions present in the original schema.

The right migration functor has a particularly simple effect on global
sections.

\begin{prop}\label{prop:pi_migration_global_sections}
Let \(F\colon\CC\to\Dc\) be a functor and let
\(X\colon\CC\to\Sets\) be a database. There is a natural bijection
\[
\operatorname{Sect}(\pi_X)
\cong
\operatorname{Sect}(\pi_{\Pi_FX}).
\]
Consequently,
\[
\mathrm{dsecat}(\pi_X)=0
\quad\Longleftrightarrow\quad
\mathrm{dsecat}(\pi_{\Pi_FX})=0.
\]
\end{prop}

\begin{proof}
Global sections of \(\pi_X\) are naturally identified with natural
transformations \(1_\CC\Rightarrow X\). Using the adjunction
\(\Delta_F\dashv\Pi_F\), we obtain
\[
\operatorname{Nat}(1_\Dc,\Pi_FX)
\cong
\operatorname{Nat}(\Delta_F1_\Dc,X).
\]
Since precomposition preserves the constant singleton database,
\(\Delta_F1_\Dc=1_\CC\). Therefore,
\[
\operatorname{Nat}(1_\Dc,\Pi_FX)
\cong
\operatorname{Nat}(1_\CC,X),
\]
which gives the required bijection. The final assertion follows because
both projections are discrete opfibrations.
\end{proof}

The preceding proposition concerns only the existence of a global section.
It does not, in general, imply that
\(\mathrm{dsecat}(\pi_X)\) and
\(\mathrm{dsecat}(\pi_{\Pi_FX})\) coincide when these invariants are strictly positive.

\subsection{Detecting initial objects through databases}\label{subsec:global_section_initial}

We conclude by characterizing directed category zero in terms of global
sections of databases. A database \(X\colon\CC\to\Sets\) is called
\emph{objectwise non-empty} if \(X(c)\neq\varnothing\) for every
\(c\in\CC\). Its category-of-elements projection is then surjective on
objects. We first construct a canonical database that detects the existence of an initial object.

\begin{lemma}\label{lem:source_database}
Let \(\CC\) be a finite connected acyclic category. Define
\(X_{\operatorname{Src}}\colon\CC\to\Sets\) by
\[
X_{\operatorname{Src}}(c)
=
\coprod_{s\in\operatorname{Src}(\CC)}
\Hom_\CC(s,c),
\qquad
X_{\operatorname{Src}}(u)(s,f)
=
(s,u\circ f).
\]
Then \(X_{\operatorname{Src}}\) is objectwise non-empty, and
\(\pi_{X_{\operatorname{Src}}}\) admits a global section if and only if
\(\CC\) has an initial object.
\end{lemma}

\begin{proof}
Every object \(c\) receives a morphism from some source: repeatedly follow
incoming non-identity morphisms until the process terminates. Hence
\(X_{\operatorname{Src}}(c)\neq\varnothing\).

Suppose that \(\pi_{X_{\operatorname{Src}}}\) admits a section. It
corresponds to elements
\(x_c=(s_c,f_c)\in X_{\operatorname{Src}}(c)\) satisfying
\[
(s_d,f_d)
=
(s_c,u\circ f_c)
\]
for every \(u\colon c\to d\). Thus \(s_c=s_d\) along every morphism.
Since \(\CC\) is connected, there is a source \(i\) such that \(s_c=i\)
for every \(c\).

We have \(x_i=(i,1_i)\). If \(f\colon i\to c\), compatibility gives
\(x_c=(i,f)\). Therefore any two morphisms \(i\to c\) coincide, while
\(x_c=(i,f_c)\) provides such a morphism. Hence \(i\) is initial.

Conversely, if \(i\) is initial and \(f_c\colon i\to c\) is the unique
morphism, then \(x_c=(i,f_c)\) defines a global section.
\end{proof}

\begin{theorem}\label{prop:dcat_zero_databases}
Let \(\CC\) be a finite connected acyclic category. The following are
equivalent:
\begin{enumerate}
    \item \(\mathrm{dcat}(\CC)=0\);
    \item every objectwise non-empty database
    \(X\colon\CC\to\Sets\) admits a global section.
\end{enumerate}
\end{theorem}

\begin{proof}
If \(\mathrm{dcat}(\CC)=0\), then
Proposition~\ref{inequality_genus_LS} gives
\(\mathrm{dsecat}(\pi_X)=0\) for every objectwise non-empty \(X\). Since
\(\pi_X\) is a discrete opfibration,
Corollary~\ref{cor:dsecat_zero_opfibration} gives a global section.

Conversely, apply the hypothesis to \(X_{\operatorname{Src}}\).
Lemma~\ref{lem:source_database} gives an initial object in \(\CC\), and
Proposition~\ref{prop:dcat_zero} yields
\(\mathrm{dcat}(\CC)=0\).
\end{proof}

Thus, for finite connected acyclic schemas, the vanishing of directed
Lusternik--Schnirelmann category is equivalent to a universal solvability
property for databases: every objectwise non-empty instance admits a
globally coherent selection. 

\bibliographystyle{plain}
\bibliography{biblio}

\end{document}